\documentclass[11pt, oneside, reqno]{amsart}
\usepackage{amsfonts, amssymb, color}
\usepackage{amsmath}
\usepackage{graphicx}
\usepackage{amscd}
\usepackage{amsthm}
\usepackage{tikz}
\usepackage{subcaption}
\usepackage{hyperref}

\usetikzlibrary{decorations.markings}
\usetikzlibrary{shapes,arrows,positioning}	
\usetikzlibrary {arrows.meta}	
\usetikzlibrary{decorations.markings,decorations.pathmorphing}
\usetikzlibrary{decorations.pathreplacing}
\usetikzlibrary{calc, fit}

\usepackage[margin=1in]{geometry}

	\newtheorem{thm}{Theorem}[section]
	
	\newtheorem{lem}[thm]{Lemma}
	\newtheorem{prop}[thm]{Proposition}
	\newtheorem{cor}[thm]{Corollary}
	
	\newtheorem{rem}[thm]{Remark}
	
	\newtheorem{Def}[thm]{Definition}
	\newtheorem{ex}[thm]{Example}

\begin{document}
	
\title[]{Some Results on Bichon's Quantum Automorphism Group of Graphs}

\author{Rajibul Haque}
\address{Department of Mathematics, Presidency University, College Street, Kolkata-700073}
\email{rajibulhaquemath@gmail.com}

\author{Ujjal Karmakar}
\address{Department of Mathematics, Presidency University, College Street, Kolkata-700073}
\email{mathsujjal@gmail.com}

\author{Arnab Mandal}
\address{Department of Mathematics, Presidency University, College Street, Kolkata-700073}
\email{arnab.maths@presiuniv.ac.in}

\maketitle

\begin{abstract}
The notion of the quantum automorphism group of a graph was introduced by J. Bichon in 2003 and T. Banica in 2005 respectively. This article explores primarily the quantum automorphism group of a graph $\Gamma$, denoted by $QAut_{Bic}(\Gamma)$, in Bichon’s framework. First, we provide a sufficient condition for non-commutativity of Bichon's quantum automorphism group and discuss several applications of this criterion. Although it is known that $QAut_{Bic}(\Gamma) \cong QAut_{Bic}(\Gamma^c)$ does not hold in general, we identify a family of graphs for which this isomorphism enforces that the graph has no quantum symmetry.
Moreover, we describe a few families of graphs having quantum symmetries whose quantum automorphism groups in Bichon’s sense are commutative. Finally, we show that free product, tensor product and free wreath product constructions can arise as Bichon's quantum automorphism groups of connected graphs in the following sense: For a finite family of compact matrix quantum groups $\{Q_i\}_{i=1}^{m}$ arising as Bichon's quantum automorphism groups of certain graphs, there exist connected graphs $\Gamma_{free}$, $\Gamma_{ten}$ and $\Gamma_{wr}$ whose quantum automorphism groups are $*_{i=1}^{m} Q_{i}$, $\otimes_{i=1}^{m} Q_i$ and $Q_1 \wr_{*} Q_2$ respectively.
\end{abstract}
 

\section{Introduction}
 The notion of the quantum symmetry group of a finite space $X_n$ was introduced by Wang in 1998 by defining the quantum permutation group $S_n^+$ \cite{Wang}, which is a special family of compact matrix quantum groups introduced earlier by Woronowicz in 1987 \cite{Wor}. Later, in 2003, Bichon first came up with the idea of the quantum symmetry group $QAut_{Bic}(\Gamma)$ of a finite graph $\Gamma$ \cite{Bichon}, which generalizes the classical automorphism group of $\Gamma$ and extends the notion of Wang’s quantum symmetry group of a finite point space. In 2005, the notion of the quantum symmetry group of a finite simple graph $\Gamma$ was further extended by Banica \cite{Ban}, denoted by $QAut_{Ban}(\Gamma)$, by establishing the inclusions $Aut(\Gamma) \subseteq QAut_{Bic}(\Gamma) \subseteq QAut_{Ban}(\Gamma) \subseteq S_n^+$ in the quantum setting. Since the above inclusion can be strict in general, the non-commutativity of $QAut_{Ban}(\Gamma)$ is more natural than that of $QAut_{Bic}(\Gamma)$. There is now a rich body of literature on Banica's quantum symmetries of graphs in many directions (\cite{Brannan, Qauttrees, asymm, Lupini, solutiongroup, SS_folded cube, SS_Advances, SSthesis} and references therein). In this article, the authors emphasize the results concerning quantum symmetries of graphs in Bichon’s framework, while several results can be extended naturally to Banica’s setting.\\
 
\noindent S. Schmidt showed that the existence of `non-trivial disjoint automorphisms' is sufficient for Banica’s quantum automorphism group $QAut_{Ban}(\Gamma)$ of a graph $\Gamma$ to be non-commutative \cite{SS_Advances}. However, the same condition does not guarantee the non-commutativity of Bichon’s quantum automorphism group by considering the 4-cycle $C_4$. In the present article, by analogously modifying that condition, we introduce the notion of `edge-free disjoint automorphisms' in order to obtain a non-commutativity criterion for $QAut_{Bic}(\Gamma)$ [Proposition \ref{Bic_noncomm}]. As applications of this criterion, we then present several non-commutativity results for Bichon's quantum automorphism group. It is known that, as in the classical case, one always has $QAut_{Ban}(\Gamma) \cong QAut_{Ban}(\Gamma^c)$, whereas in Bichon's framework there exist graphs $\Gamma$ for which $QAut_{Bic}(\Gamma) \ncong QAut_{Bic}(\Gamma^c)$. Thus, it may happen that one of them is commutative while the other is not; for instance, one may consider the complete graphs $K_n$ for $n \geq 4$. By applying Proposition \ref{Bic_noncomm}, we produce infinitely many graphs $\Gamma$ such that both $QAut_{Bic}(\Gamma)$ and $QAut_{Bic}(\Gamma)$ are non-commutative. Similarly, it may also occur that a graph $\Gamma$ has non-commutative $QAut_{Bic}(\Gamma)$, whereas its line graph $L(\Gamma)$ does not. However, it is shown in Proposition \ref{L(T)_noncomm} that a graph $\Gamma$ and its line graph $L(\Gamma)$ can both simultaneously possess non-commutative Bichon's quantum automorphism groups under certain hypotheses. Consequently, it follows that for almost all trees $T$, both $QAut_{Bic}(T)$ and $QAut_{Bic}(L(T))$ are non-commutative. We also discuss a necessary and sufficient condition for $QAut_{Bic}(K_{m,n})$ to be non-commutative. Moreover, some graphs with non-commutative Bichon's quantum automorphism groups can be obtained via familiar graph products, and we verify whether the analogue of Proposition 2.2.1 in \cite{SSthesis} holds in Bichon’s framework. We find that the analogous result holds for the Cartesian, direct, and tensor products, but cannot be extended for the lexicographic product.\\ 
On the other hand, it is evident that if a graph $\Gamma$ has no quantum symmetry (i.e., $QAut_{Ban}(\Gamma)$ is commutative), then $QAut_{Bic}(\Gamma) \cong QAut_{Bic}(\Gamma^c)$, though its converse is not true (see Examples \ref{example_self-comp}, \ref{Ex_graph_G'}, \ref{Ex_graph_G'_2}). However, it has been shown that for the family of graphs whose members or their complements do not contain a quadrangle (for instance, forests), the converse holds. Thus, within this family, if both $QAut_{Bic}(\Gamma)$ and $QAut_{Bic}(\Gamma^c)$ are commutative, then $\Gamma$ has no quantum symmetry. This observation motivates us to find examples of graphs $\Gamma$ having quantum symmetries while both $QAut_{Bic}(\Gamma)$ and $QAut_{Bic}(\Gamma^c)$ are commutative. We construct infinitely many such graphs with this property. Furthermore, a few examples of such graphs can also be constructed using graph products. We establish that if the surjective $*$-homomorphisms in Proposition 2.2.1 (i), (iii) of \cite{SSthesis} are isomorphisms in Banica’s framework, then the corresponding $*$-homomorphisms for Bichon's quantum automorphism groups (see Proposition \ref{Bic_noncomm_prod_graph}) are also isomorphisms. Consequently, by carefully choosing graphs that satisfy the eigenvalue-related conditions as in Theorem 2.2.2 (i), (iii) of \cite{SSthesis} for the Cartesian and strong products of graphs, it is possible to construct examples that admit quantum symmetries while having commutative Bichon's quantum automorphism groups.\\

Finally, we address the question of whether certain quantum group product constructions between quantum symmetry groups of graphs can arise as the quantum symmetries of some connected graphs. Specifically, we focus on the free product ($*$), tensor product ($\otimes$), and free wreath product ($\wr_{*}$) of compact matrix quantum groups. Recently, it was shown in \cite{Qauttrees} that given a finite family of trees ${T_i}$, there exists a tree $T$ such that $QAut_{Ban}(T) \cong *_{i=1}^{m} ~~ QAut_{Ban}(T_i)$, and consequently, it follows from Theorem 3.5 of \cite{SS_Advances} that  $QAut_{Bic}(T) \cong *_{i=1}^{m} ~~ QAut_{Bic}(T_i)$. In the present article, we establish that for a finite family of Bichon’s quantum symmetry groups $\{Q_i\}_{i=1}^{m}$ of certain graphs $\{\Gamma_i\}_{i=1}^{m}$, there exist connected graphs $\Gamma_{free}$ and $\Gamma_{ten}$ whose quantum symmetry groups are $\ast_{i=1}^{m} Q_i$ and $\otimes_{i=1}^{m} Q_i$, respectively, in Bichon’s framework. Moreover, the construction of $\Gamma_{free}$ also ensures the validity of the same result in Banica’s setting. In addition, we show that there exists a connected graph $\Gamma_{wr}$ whose quantum symmetry group is $Q_1 \wr_{*} Q_2$. In Banica’s framework, the existence of such a $\Gamma_{wr}$ is more or less immediate from lexicographic product constructions (see Theorem 1.1 of \cite{Lexico}). However, the same approach does not work in Bichon’s framework. Instead, we use the corona product \cite{corona} to construct a connected graph $\Gamma_{wr}$, which works simultaneously in both Banica’s and Bichon’s settings.\\

The presentation of the article is as follows: In the second section, we recall the required terminologies along with several graph products for finite graphs. Next, we briefly review compact matrix quantum groups, including examples and constructions such as the free product, tensor product, and free wreath product. We then mention relevant results from the literature regarding the quantum symmetry of finite graphs, with emphasis on those that will be applied later in the article. In the third section, we provide a sufficient condition for the non-commutativity of Bichon’s quantum automorphism groups of graphs (Proposition \ref{Bic_noncomm}) and discuss some of its applications. In Proposition \ref{Forest_Bic}, we show that the converse of Proposition \ref{Bic_noncomm} also holds within the class of finite forests. We also examine the non-commutative behaviour of quantum symmetry groups under standard graph products in Bichon’s setting (Proposition \ref{Bic_noncomm_prod_graph}).
The fourth section focuses on commutativity results for Bichon’s quantum automorphism groups. Proposition \ref{Bic_comp_equal} establishes that, for a certain family of graphs $\Gamma$, the property $QAut_{Bic}(\Gamma) \cong QAut_{Bic}(\Gamma^c)$ enforces that $\Gamma$ has no quantum symmetry. Furthermore, Proposition \ref{disjont_union_graphs}, Examples \ref{Ex_graph_G'} and \ref{Ex_graph_G'_2} present infinite families of graphs whose quantum symmetry groups are commutative. Consequences of Proposition \ref{Prod_Ban_iso_imply_Bic_iso} further ensure that such graphs can also be constructed using graph products.
Finally, the last section is divided into three subsections. In the first, we show that given a finite family of finite graphs $\{\Gamma_i\}_{i=1}^{m}$, there exists a connected graph $\Gamma_{free}$ such that $QAut_{Bic}(\Gamma_{free}) \cong *_{i=1}^{m} ~~ QAut_{Bic}(\Gamma_i)$. The next two subsections establish analogous results for the tensor product and the free wreath product, respectively.

\section{Preliminaries}
\subsection{Notations and conventions}
For a set $X$, the cardinality of $X$ will be denoted by $|X|$ and  the identity function on $X$ will be denoted by $id_{X}$. For $n \in \mathbb{N}$, $[n]$ denotes the set $\{1, 2,..., n\}$. For a $C^*$-algebra $ \mathcal{C}$, $\mathcal{C}^*$  is the set of all linear bounded functionals on $ \mathcal{C} $. For a set $A$, span($A$) will denote the linear space spanned by the elements of $A$. The tensor product `$\otimes$' and `$\otimes_{max}$' denote the spatial (or minimal) tensor product and maximal tensor product between $C^*$-algebras respectively. The matrix $I_n$ denotes the identity matrix of  $M_n(\mathbb{C})$. The matrices $0_{m,n}$ and $J_{m,n} \in M_{m \times n}(\mathbb{C})$ denote the matrices with all entries equal to 0 and 1, respectively; in particular, we write $0_{n}:=0_{n,n}$ and $J_{n}:= J_{n,n}$. According to our convention, the tensor product $A \otimes B$ of two algebra-valued matrices $A = (a_{ij})_{m \times n}$ and $B = (b_{kl})_{p \times q}$ is the block matrix $(A b_{kl})_{k \in [p], l \in [q]}$ of order $mp \times nq$, where each block $A b_{kl}$ is given by $(a_{ij} b_{kl})_{i \in [m], j \in [n]}$.\\
 All the $C^*$-algebras considered are assumed to be unital. 

\subsection{Finite Graphs}
 A {\bf simple finite graph} $\Gamma$ is a pair $(V(\Gamma), E(\Gamma))$  consisting of  a finite set $V(\Gamma)$ of vertices and a finite set $E(\Gamma) \subseteq V(\Gamma) \times V(\Gamma) $ of edges such that $(v,v) \notin E(\Gamma)$, and $(v,w) \in E(\Gamma)$ if and only if $(w,v) \in E(\Gamma)$. The {\bf order} of a graph $\Gamma$ is defined as the cardinality of its vertex set $V(\Gamma)$. Two vertices $v,w \in V(\Gamma)$ are said to be {\bf adjacent} if $(v,w) \in E(\Gamma)$. {\bf Throughout this article, we restrict ourselves to simple finite graphs, which we will often refer to simply as finite graphs}. For a given $v \in V(\Gamma)$, a vertex $u \in V(\Gamma)$ is called the {\bf neighbour} of $v \in V(\Gamma)$ if $(u,v) \in E(\Gamma)$, and $N_{\Gamma}(v):= \{u \in V(\Gamma): (u,v) \in E(\Gamma)\}$. The {\bf degree} of a vertex $v$ in $\Gamma$, denoted by $deg_{\Gamma}(v)$, is the cardinality of $N_{\Gamma}(v)$. If the graph is already specified in the context, we sometimes simply denote it by $\deg(v)$. A vertex $v$ is called {\bf isolated} if $deg(v)=0$, i.e. $v$ has no adjacent vertices in $\Gamma$. A \textbf{path} $\alpha$ of length $n$ (denoted by $P_n$) in a simple graph $\Gamma$ is a sequence $\alpha=v_{0}v_{1} \cdots v_{n} $ of vertices in $V(\Gamma)$ such that $(v_i, v_{i+1}) \in E(\Gamma)$ for $0 \leq i < n $, and $\alpha$ is said to be a path joining $v_0$ and $v_n$. A {\bf $n$-cycle} (denoted by $C_n$) is a path of length $n$ with $v_{0}=v_{n}$. In particular, a $4$-cycle is called a {\bf quadrangle}.
 A graph is said to be {\bf connected} if for any two distinct vertices $v,w \in V(\Gamma)$, either $(v,w) \in E(\Gamma)$ or there exists a path joining $v$ and $w$.
 A {\bf tree} is a connected graph that contains no cycles. More generally, a {\bf forest} is a graph that contains no cycles, i.e. a graph each of whose connected components is a tree. A {\bf complete graph} with $n$ vertices, denoted by $K_n$, contains $n$ distinct vertices where every pair of distinct vertices are adjacent. A {\bf bipartite graph} $B(M,N)$ is a graph whose vertex set can be partitioned into two disjoint sets $M$ and $N$ such that two vertices in $V(B(M,N))$ are adjacent only if one belongs to $M$ and the other belongs to $N$. A {\bf complete bipartite graph} $K_{m,n}$ is a bipartite graph $B(M,N)$ with $|M|=m$ and $|N|=n$ such that every vertex $v \in M$ is adjacent to every vertex $w \in N$.   
 In a connected graph $\Gamma$, the {\bf distance between two vertices} $v,w \in V(\Gamma)$, denoted by $d(v,w)$, is the length of the shortest path joining $v$ and $w$.\\
 The {\bf adjacency matrix} of $\Gamma$ (denoted by $A_{\Gamma}$) with respect to the ordering of the vertices $ (v_{1},v_{2},..., v_{n}) $ is a symmetric matrix $ (a_{ij})_{n \times n} $ with 
 $ a_{ij} =
 \begin{cases}
1 & if ~~ (v_{i}, v_{j}) \in E(\Gamma) \\
 0 &  ~~ otherwise 
 \end{cases} $, where $n:= |V(\Gamma)|$.\\
 The {\bf complement of a graph} $\Gamma$, denoted by $\Gamma^c$, is a simple finite graph with vertex set $V(\Gamma^c):=V(\Gamma)$ and edge set $E(\Gamma^c):= [V(\Gamma) \times V(\Gamma)] \setminus [E(\Gamma) \sqcup L]$, where $L:= \{(v,v): v \in V(\Gamma)\}$. Observe that if $|V(\Gamma)|=n$, then  $A_{\Gamma^c}=J_n-I_n-A_{\Gamma}$.\\
 Given a graph $\Gamma$, its {\bf line graph} $L(\Gamma)$ is a graph whose vertex set $V(L(\Gamma))$ corresponds to the edge set $E(\Gamma)$ of $\Gamma$ and two vertices $e,f \in V(L(\Gamma))$ are adjacent if and only if their corresponding edges in $E(\Gamma)$ share a common endpoint.
 
 Let $\Gamma_1$ and $\Gamma_2$ be simple finite graphs. A {\bf (graph) isomorphism} between $\Gamma_1$ and $\Gamma_2$ is a bijection $\sigma: V(\Gamma_1) \to V(\Gamma_2)$ such that $(v,w) \in E(\Gamma_1) \iff (\sigma(v), \sigma(w)) \in E(\Gamma_2)$.  Each isomorphism $\sigma$ can be viewed as a permutation matrix in $GL_{n}(\mathbb{C})$ (where $n:=|V(\Gamma_1)|=|V(\Gamma_2)|$) associated with the permutation $\sigma$ of $S_n$ that satisfies the relation $ \sigma A_{\Gamma_1}= A_{\Gamma_2} \sigma$. A {\bf (graph) automorphism} of $\Gamma$ is a (graph) isomorphism from $\Gamma$ to itself (i.e. a bijection  $\sigma: V(\Gamma) \to V(\Gamma)$ such that  $(v,w) \in E(\Gamma) \iff (\sigma(v), \sigma(w)) \in E(\Gamma)$). The set of all automorphisms of $\Gamma$ forms a group under composition, called the automorphism group  $Aut(\Gamma)$ of $\Gamma$.  Under the above representation, we have $Aut(\Gamma)=\{\sigma \in S_n : \sigma A_{\Gamma}= A_{\Gamma} \sigma\} \subset  S_n$.
 A graph $\Gamma_1$ is said to be isomorphic to a graph $\Gamma_2$  (denoted by $\Gamma_1 \cong \Gamma_2$) if there exists an isomorphism between $\Gamma_1$ and $\Gamma_2$. A graph is said to be 
 {\bf self-complementary} if $\Gamma$ is isomorphic to $\Gamma^c$.\\

 
 We recall the following four graph products, which are already well-studied in the literature (consult \cite{graph product} for details).
 \begin{Def} \label{Def_Prod_graph}
 	Let $\Gamma_1$ and $\Gamma_2$ be finite graphs of order $n$ and $m$ respectively. 
 	\begin{enumerate}
 		\item[(i)] The {\bf Cartesian product} $\Gamma_1 \square \Gamma_2$ is the graph with vertex set $V(\Gamma_1) \times V(\Gamma_2)$, and 
 		$((i,\alpha), (j, \beta)) \in E(\Gamma_1 \square \Gamma_2)$ if and only if $(i=j \text{ and } (\alpha, \beta) \in E(\Gamma_2))
 		\text{ or } ((i,j) \in E(\Gamma_1) \text{ and } \alpha = \beta)$. Therefore, the adjacency matrix is of the form:
 		$$A_{\Gamma_1 \square \Gamma_2} = A_{\Gamma_1} \otimes I_m + I_n \otimes A_{\Gamma_2}.$$
 		\item[(ii)] The {\bf direct product} $\Gamma_1 \times \Gamma_2$ is the graph with vertex set $V(\Gamma_1) \times V(\Gamma_2)$, and 
 		$((i,\alpha), (j, \beta)) \in E(\Gamma_1 \times  \Gamma_2)$ if and only if $( (i,j)\in E(\Gamma_1) \text{ and } (\alpha, \beta) \in E(\Gamma_2)).$ Hence, we obtain:
 		$$A_{\Gamma_1 \times \Gamma_2}= A_{\Gamma_1} \otimes A_{\Gamma_2}.$$ 
 		\item[(iii)] The {\bf strong product} $\Gamma_1 \boxtimes \Gamma_2$ is the graph with vertex set $V(\Gamma_1) \times V(\Gamma_2)$, and 
 		$((i,\alpha), (j, \beta)) \in E(\Gamma_1 \boxtimes \Gamma_2)$ if and only if one of the following holds:
 		 \begin{itemize}
 		 	\item $(i=j \text{ and } (\alpha, \beta) \in E(\Gamma_2))$,
 		 	\item $((i,j) \in E(\Gamma_1) \text{ and } \alpha = \beta)$,
 		 	\item $( (i,j)\in E(\Gamma_1) \text{ and } (\alpha, \beta) \in E(\Gamma_2))$.
 		 \end{itemize}
       Therefore,  $$ A_{\Gamma_1 \boxtimes \Gamma_2} = [(A_{\Gamma_1} + I_n ) \otimes  (A_{\Gamma_2} + I_m )] - I_{mn} .$$ 
         \item[(iv)] The {\bf lexicographic product} \footnote{Throughout this article, we follow the convention for the lexicographic product given in Definition 1.2.13 (iv) of \cite{SSthesis}.} $\Gamma_1 \circ \Gamma_2$ is the graph with vertex set $V(\Gamma_1) \times V(\Gamma_2)$, and 
         $((i,\alpha), (j, \beta)) \in E(\Gamma_1 \circ  \Gamma_2)$ if and only if $(\alpha, \beta) \in E(\Gamma_2)$ or $((i,j) \in E(\Gamma_1) \text{ and } \alpha = \beta)$. This yields $$A_{\Gamma_1 \circ \Gamma_2} = A_{\Gamma_1} \otimes  I_m + J_n \otimes A_{\Gamma_2}.$$  \\ 
         Note that if $X_n$ denotes the graph consisting of $n$ isolated vertices, then $\Gamma \circ X_n$ is the disjoint union of $n$ copies of $\Gamma$, simply denoted by $n\Gamma$.   	 
 	\end{enumerate}
 \end{Def} 
 
 \noindent In 1970, Frucht and Harary introduced another graph product, called the corona product \cite{corona}, which plays a crucial role in this article in our study of Bichon's quantum symmetries of graphs and is defined as follows:
 \begin{Def} \label{Def_corona_prod}
 	Given two graphs $\Gamma_1=(V(\Gamma_1), E(\Gamma_1))$ and $\Gamma_2=(V(\Gamma_2), E(\Gamma_2))$, their {\bf corona product} \footnote{The convention for the corona product, as introduced in \cite{corona}, is used in this article.} $\Gamma_1 \odot \Gamma_2$ is a new graph with vertex set $V(\Gamma_1 \odot \Gamma_2):=V(\Gamma_1) \sqcup \left[V(\Gamma_2) \times V(\Gamma_1)\}\right]$ and edge set $E(\Gamma_1 \odot \Gamma_2):= E(\Gamma_1) \sqcup E' \sqcup E'' $, where $E':= \bigsqcup \limits_{v \in V(\Gamma_1)} \{((w,v), (w',v)): (w,w') \in E(\Gamma_2)\}$ and $E'' := \{((w,v), v) : w \in V(\Gamma_2) \text{ and } v \in V(\Gamma_1) \}$.
 \end{Def}
 \noindent For example, $K_2 \odot K_1$ is isomorphic to $P_3$, and the graph $C_3 \odot P_2$ is shown in Figure \ref{C3_corona_P2}.
 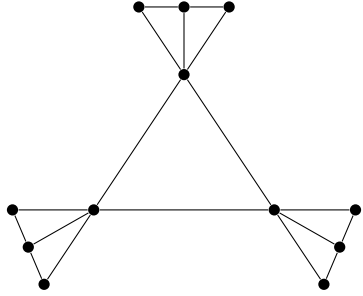
\begin{figure}[htpb]
 	\centering
 	\begin{tikzpicture}[scale=0.6]
 	\tikzset{dot/.style={circle, fill, inner sep=1.5pt}}
 	
 	\node[dot] (v1) at (0, 2) {};
 	\node[dot] (v2) at (-2, -1) {};
 	\node[dot] (v3) at (2, -1) {};
 	
 	\node[dot] (t1) at (-1, 3.5) {};
 	\node[dot] (t2) at (1, 3.5) {};
 	\node[dot] (t3) at (0, 3.5) {};
 	
 	\node[dot] (l1) at (-3.8, -1) {};
 	\node[dot] (l2) at (-3.1, -2.65) {};
 	\node[dot] (l3) at (-3.45, -1.825) {};
 	
 	\node[dot] (r1) at (3.8, -1) {};
 	\node[dot] (r2) at (3.1, -2.65) {};
 	\node[dot] (r3) at (3.45, -1.825) {};
 	
 	\draw (v1) -- (v2) -- (v3) -- (v1) -- cycle;
 	\draw (v1) -- (t1) --  (t3) -- (t2) -- (v1) -- (t3)-- cycle;
 	\draw (v2) -- (l1) -- (l3) -- (l2) --(v2) --(l3) -- cycle;
 	\draw (v3) -- (r1) -- (r3) -- (r2) -- (v3) --(r3) -- cycle;	
 	\end{tikzpicture}
 	\caption{$C_3 \odot P_2$}   \label{C3_corona_P2}
 \end{figure}

 \subsection{Compact quantum groups }
 In this subsection, some important aspects of compact (matrix) quantum groups and quantum symmetry groups are briefly recalled. We refer the readers to \cite{Van}, \cite{Wang}, \cite{Wor}, \cite{webcqg}, \cite{remark}, \cite{universalCQG}, \cite{Tim}, \cite{Nesh}, \cite{Fres} for further details.
 
 \begin{Def} \label{CQG}
 	A compact quantum group (CQG) is a pair $(\mathcal{Q}, \Delta) $, where $\mathcal{Q}$ is a unital $C^{*}$-algebra and the coproduct $ \Delta : \mathcal{Q} \to \mathcal{Q} \otimes \mathcal{Q} $  is a unital $C^{*}$-homomorphism such that \vspace{0.1cm}
 	\begin{itemize}
 		\item[(i)] $(id_{\mathcal{Q}} \otimes \Delta)\Delta = (\Delta \otimes id_{\mathcal{Q}})\Delta $. \vspace{0.1cm}
 		\item[(ii)] $ span\{ \Delta(\mathcal{Q})(1 \otimes \mathcal{Q} )\}$ and $ span\{\Delta(\mathcal{Q})(\mathcal{Q} \otimes 1)\} $ are dense in $(\mathcal{Q}\otimes \mathcal{Q})$. \vspace{0.2cm}
 	\end{itemize}
 	Given two CQGs $(\mathcal{Q}, \Delta)$ and $(\mathcal{Q}', \Delta') $, a compact quantum group morphism (CQG morphism) from $(\mathcal{Q}, \Delta)$ to $(\mathcal{Q}', \Delta')$ is a $C^{*}$-homomorphism $ \phi: \mathcal{Q} \to \mathcal{Q}' $ such that $ (\phi \otimes \phi)\Delta=\Delta' \phi $ .\\
 \end{Def}
 
\noindent The following definition introduces a special family of CQGs, which will exclusively appear in this article.
 \begin{Def} \label{CMQG}
 	Let $A$ be a unital $C^*$-algebra, and let $q=(q_{ij})_{n \times n}$ be a matrix whose entries are coming from $A$. $(A,q)$ is called a compact matrix quantum group (CMQG) if the following conditions are satisfied:
 	\begin{enumerate}
 		\item[(i)] A is generated by $q_{ij}$'s,
 		\item[(ii)] both $q$ and ${q}^t:=(q_{ji})_{n \times n}$ are invertible in $M_n(A)$,
 		\item[(iii)] there exists a unital $C^*$-homomorphism $\Delta: A \to A \otimes A $, referred to as the coproduct, such that $\Delta(q_{ij})= \sum_{k=1}^{n}q_{ik} \otimes q_{kj}$ for all $i,j \in \{1,2,...,n\}$. 
 	\end{enumerate}
 	The matrix $q$ is called the fundamental representation of $A$. 
 \end{Def}
It can be shown that every CMQG is a CQG. In a CMQG $(A,q)$, if $A_0$ denotes a *-subalgebra generated by $\{q_{ij}\}_{i,j \in [n]}$, then there exists a linear anti-multiplicative map $\kappa: A_0 \to A_0$, called the {\bf antipode}, such that $\kappa(q_{ij})=(q^{-1})_{ij}$, where $q^{-1}$ is the inverse of $q$.\\
\noindent We sometimes denote the CMQG simply by $G$ and the underlying $C^*$-algebra associated to it by $C(G)$, i.e. $G=(C(G),q)$.
 
 \begin{Def} \label{identical}
 	Two CMQGs $(A,q)$ and  $(A',q')$, where $A$ and $A'$ are unital $C^*$-algebras with fundamental representations $q=(q_{ij})_{n \times n}$ and $q'=(q_{ij}')_{n \times n}$ respectively, are said to be { \bf identical} (denoted by $(A,q) \approx (A',q')$) if there exists a $C^*$-isomorphism $ \phi: A \to A' $ such that $ \phi(q_{ij})=q_{ij}'$.\\
 	More generally, two CQGs $(A, \Delta)$ and $(A', \Delta')$ (in particular, CMQGs $(A,q)$ and  $(A',q')$) are said to be {\bf CQG-isomorphic} (simply denoted by $A \cong A'$) if there exists a $C^{*}$-isomorphism $ \phi: A \to A' $ such that $ (\phi \otimes \phi)\Delta=\Delta' \phi $.
 \end{Def}
 \begin{Def}
 	A CMQG $(A,q)$ is said to be a {\bf quantum subgroup} of $(A',q')$ if there exists a surjective $C^*$-homomorphism $ \phi: A' \to A $ such that $ \phi(q'_{ij})=q_{ij}$.    
 \end{Def}
\begin{ex}  \label{Example_CQG} 
 Now, let us review some examples and constructions of CMQGs, which we will utilize later in this article.
 \begin{enumerate}
 	\item Consider the universal $C^{*}$-algebra generated by $\{u_{ij}: i,j \in \{1,2,...,n\}\}$ such that the following `magic unitary' relations are satisfied: \vspace{0.1cm}
 	\begin{itemize}
 		\item  $u_{ij}^{2}=u_{ij}=u_{ij}^{*} $ for all $i,j \in \{1,2,...,n\}$, \vspace{0.1cm}
 		\item  $\sum_{k=1}^{n} u_{ik}= \sum_{k=1}^{n} u_{kj}=1 $ for all $i,j \in \{1,2,...,n\}$. \vspace{0.1cm}
 	\end{itemize}
 	and denote it by $C(S_n^+)$.\\  
 Then $S_n^+ =(C(S_n^+),u)$ is a CMQG, called the quantum permutation group. Here, the matrix  $u$ is invertible with $u^{-1}=u^*:=(u_{ji}^*)_{n \times n}$, and so the antipode $\kappa$ is given by $\kappa(u_{ij})=u_{ji}$. \\
 	Moreover,  $C(S_n^+)$ is non-commutative iff $n \geq 4$ and 
 	$$C(S_n)=C(S_n^+)/<u_{ij}u_{kl}-u_{kl}u_{ij}>.$$
 	This quantum group was introduced in 1998 by Wang as the quantum automorphism group $QAut(X_n)$ of a finite $n$-point space $X_n$ (consult \cite{Wang}, \cite{BBC} for more details).  \\
 	\item $C(H_{n}^{+})$ is defined to be the universal $C^*$-algebra generated by $\{u_{ij}: i,j \in \{1,2,...,n\}\}$ such that \vspace{0.1cm}
 	\begin{itemize}
 		\item $u_{ij}=u_{ij}^*$, \vspace{0.1cm}
 		\item the matrix $u=(u_{ij})_{n \times n}$ is orthogonal, \vspace{0.1cm}
 		\item $u_{ik} u_{jk} = u_{ki} u_{kj} = 0$ for $i \neq j$. \vspace{0.1cm}
 	\end{itemize}
 	 Then $(C(H_{n}^{+}), u)$ forms a CMQG and it's denoted by $H_{n}^{+}$ (see \cite{Bichonwreath} for details).\\
 	 \item Let $G=(A,q)$ and $G'=(A',q')$ be CMQGs with multiplicative identity $1_{A}$ and $1_{A'}$ respectively. Then their free product (denoted by $A*A'$) is a CMQG $(A*A', q \oplus q')$ \cite{Wangfree}, whose underlying $C^*$-algebra $A*A'$ is the unital universal $C^*$-algebra generated by $1$, $\{q_{ij}\}$ and $\{q'_{kl}\}$
 	 such that the $q_{ij}$'s and $q'_{kl}$'s fulfill the relations of $A$ and $A'$ respectively, and $1_{A}=1=1_{A'}$. Here the fundamental representation $q \oplus q'$ denotes the block diagonal matrix $\begin{bmatrix}
 	 q & 0 \\0 & q' \\
 	 \end{bmatrix}$.\\
 	 More generally, for a finite family of CMQGs $\{(A_i, q^{i})\}_{i \in [n]}$, their free product can be defined similarly and is denoted by $(*_{i=1}^{n} ~A_i, \oplus_{i=1}^{n} ~ q^{i})$ . In the special case where all $A_i$'s are equal to $A$, the free product is denoted by $(*_n ~ A, \oplus_{i} ~ q^i)$.\\
 	 
 	 \item Let $G=(A,q)$ and $G'=(A',q')$ be CMQGs with multiplicative identity $1_{A}$ and $1_{A'}$ respectively. Then their tensor product (simply denoted by $A \otimes A'$) is a CMQG $(A \otimes_{max} A', q \oplus q')$, whose underlying $C^*$-algebra $A \otimes_{max} A'$ is the unital universal $C^*$-algebra (with respect to maximal $C^*$-norm) generated by $1$, $q_{ij}$'s and $q'_{kl}$'s 
 	 such that $q_{ij}$'s and $q'_{kl}$'s fulfill the relations of $A$ and $A'$ respectively, and $1_{A}=1=1_{A'}$. Additionally, $q_{ij} q'_{kl} = q'_{kl} q_{ij}$ for all $i,j,k,l$. Again, the fundamental representation $q \oplus q'$ denotes the block diagonal matrix $\begin{bmatrix}
 	 q & 0 \\0 & q' \\
 	 \end{bmatrix}$.\\
 	 Alternatively, one may consider the fundamental representation $q \otimes q':= (q_{ij}q'_{\alpha \beta})_{(i \alpha, j \beta)}$ over the same $C^*$-algebra $A \otimes_{max} A'$. In this case as well, $(A \otimes_{max} A', q \otimes q')$ forms a CMQG (consult \cite{Wang tensor}, \cite{Fres} for further details).\\
 	 \item Let $(A,q)$ be a CMQG with fundamental representation $q$ of order $m \times m$, and let $(Q,u)$ be a quantum subgroup of a quantum permutation group of $S_n^+$. The free wreath product of $(A,q)$ with $(Q,u)$ is a CMQG, simply denoted by $A \wr_{*} Q$, whose underlying $C^*$-algebra is
 	 $$(*_n A) * Q /<q_{ij}^{\alpha}u_{\alpha \beta}-u_{\alpha \beta}q_{ij}^{\alpha}>,$$ where $q_{ij}^{\alpha}$'s $(i,j \in [m])$ denotes the generators of the $\alpha$-th copy of $A$ in the free product $(*_n ~A, \oplus_{\alpha \in [n]} ~q^{\alpha})$. The canonical fundamental representation $(q \wr_{*} u)$ is then given by the block matrix $(q^{\alpha}u_{\alpha \beta})_{\alpha, \beta \in [n]}$, i.e. $(q \wr_{*} u)_{i \alpha, j \beta} := q^{\alpha}_{ij} u_{\alpha \beta}$.\\
 	 For example, it can be shown that $H_{n}^{+}$ is identical to $C(S_2) \wr_{*} S_n^+$. For further details about the free wreath product $\wr_{*}$ of a CMQG with a quantum permutation group, we refer the reader to \cite{Bichonwreath}, \cite{Free prod formulae}, \cite{lexico}.
 \end{enumerate}	 
	
\end{ex}

 \subsection{Quantum Automorphism Group of a Graph}
 In 1998, Wang introduced the quantum automorphism group of a finite space $X_n$ in \cite{Wang} by showing that $S_n^+$ is the ``largest" CQG that acts faithfully on the finite-dimensional $C^*$-algebra $C(X_n)$ associated to a finite $n$-point space $X_n$.  
Later, the notion was further extended by Bichon \cite{Bichon} and Banica \cite{Ban} independently to the setting of the quantum automorphism groups of finite graphs.

\begin{Def}[Banica, 2005] \label{Def_Banica}
	Let $\Gamma=(V(\Gamma),E(\Gamma))$ be a finite graph on $n$ vertices with  $V(\Gamma)=\{1,2,...,n\}$. The quantum automorphism group of $\Gamma$ (in the sense of Banica), denoted by $QAut_{Ban}(\Gamma)$,  is the compact matrix quantum group whose underlying $C^*$-algebra is the universal $C^{*}$-algebra with generators $u_{ij}$ $(1 \leq i, j \leq n)$ satisfying the relations
	\begin{align}
	& u_{ij}=u_{ij}^{*}=u_{ij}^{2} \label{proj}, \\
	& \sum_{l=1}^{n}u_{il}=1=\sum_{l=1}^{n}u_{li}, \label{row_sum_1}\\
	& u A_{\Gamma}=A_{\Gamma} u \label{comm_with_adj},
	\end{align}
	and $u=(u_{ij})_{n \times n}$ is the fundamental representation.
\end{Def}

\begin{Def}[Bichon, 2003] \label{Def_Bichon}
Let $\Gamma=(V(\Gamma),E(\Gamma))$ be a finite graph on $n$ vertices with  $V(\Gamma)=\{1,2,...,n\}$. The quantum automorphism group of $\Gamma$ (in the sense of Bichon), denoted by $QAut_{Bic}(\Gamma)$,  is the compact matrix quantum group whose underlying $C^*$-algebra is the universal $C^{*}$-algebra with generators $u_{ij}$ $(1 \leq i, j \leq n)$ satisfying Relations \eqref{proj}, \eqref{row_sum_1}, \eqref{comm_with_adj}, and additionally, 
\begin{align} \label{Bic_comm_relation}
u_{ij}u_{kl}=u_{kl}u_{ij} \text{ whenever } (i,k) \in E(\Gamma) \text{ and } (j,l) \in E(\Gamma).
\end{align}
The matrix $u=(u_{ij})_{n \times n}$ is  the fundamental representation of the CMQG  $QAut_{Bic}(\Gamma)$.
\end{Def}
For simplicity, we will denote both the CMQG defined in Definition \ref{Def_Banica} (resp. \ref{Def_Bichon}) and its underlying $C^*$-algebra by the same notation $QAut_{Ban}(\Gamma)$ (resp. $QAut_{Bic}(\Gamma)$). Moreover, the fundamental representations $u$ associated with the generators will be referred to as the {\bf canonical fundamental representations of $QAut_{Ban}(\Gamma)$ and $QAut_{Bic}(\Gamma)$} respectively.

\begin{prop} \label{adj_nonadj_prod=0}
Let $u$ be the fundamental representation of $QAut_{Ban}(\Gamma)$. Then Relation \eqref{comm_with_adj} is equivalent to the following relations:
\begin{align}
u_{ij}u_{kl}=0=u_{kl}u_{ij} \text{ if } (i,k) \in E(\Gamma) \text{ and } (j,l) \notin E(\Gamma), \label{adj_nonadj_prod=0_eq1}\\
u_{ij}u_{kl}=0=u_{kl}u_{ij} \text{ if } (i,k) \notin E(\Gamma) \text{ and } (j,l) \in E(\Gamma). \label{adj_nonadj_prod=0_eq2}
\end{align}
\end{prop}
 
\noindent The following theorem ensures that the quantum automorphism group of a graph always contains its classical automorphism group and thereby motivates the terminology that a graph has no quantum symmetry.	
\begin{thm}{[Theorem 2.1.6, \cite{SSthesis}]} \label{QAut_comm}
	Let $\Gamma$ be a finite graph, and let $A:=QAut_{Ban}(\Gamma)/<u_{ij}u_{kl}-u_{kl}u_{ij}>$ be the quotient of $QAut_{Ban}(\Gamma)$ by the commutator ideal. Then the compact matrix quantum group $(A,u)$ is identical to $(C(Aut(\Gamma)), \chi)$, where $\chi:=(\chi_{ij})_{n \times n}$ is the standard fundamental representation of $C(Aut(\Gamma))$ defined by $\chi_{ij}(\sigma)=\sigma_{ij}$ for all $\sigma \in Aut(\Gamma)$.
\end{thm}

\begin{Def}
	For a finite graph $\Gamma$, we say that $\Gamma$ has no quantum symmetry if $QAut_{Ban}(\Gamma)$ is commutative, i.e. $(QAut_{Ban}(\Gamma),u) \approx (C(Aut(\Gamma)), \chi)$.
\end{Def}
\noindent In the view of Definitions \ref{Def_Banica}, \ref{Def_Bichon} and Theorem \ref{QAut_comm}, there exist surjective *-homomorphisms
$$C(S_n^+) \to QAut_{Ban}(\Gamma) \to QAut_{Bic}(\Gamma) \to C(Aut(\Gamma))$$ such that $u_{ij} \mapsto u_{ij} \mapsto u_{ij} \mapsto \chi_{ij}$. Clearly, if $\Gamma$ has no quantum symmetry, then $QAut_{Bic}(\Gamma)$ is commutative and $(QAut_{Bic}(\Gamma), u) \approx (C(Aut(\Gamma)), \chi)$. \\

\noindent Next, we recall a few useful results that will be used later in this article.
\begin{lem}{[Lemma 3.1, \cite{SS_petersen}]} \label{6}
	Let $\Gamma$ be a finite graph, and let $u$ be the canonical fundamental representation of ${QAut}_{Ban}(\Gamma)$. If 
	$$ u_{ij} u_{kl} = u_{ij} u_{kl} u_{ij}, $$
then $u_{ij}$ and $u_{kl}$ commute.	
\end{lem}

\begin{prop}{[Lemma 3.2.3, \cite{Fulton}]} \label{3}
	Let $\Gamma$ be a finite graph, and let $u$ be the canonical fundamental representation of ${QAut}_{Ban}(\Gamma)$. If $deg(i) \neq deg(j)$ for some $i,j \in V(\Gamma)$, then $u_{ij} = 0$.	
\end{prop}

\begin{prop}{[Lemma 4.3, \cite{solutiongroup}]} \label{2}	
	Let $\Gamma$ be a finite graph with $v,w \in V(\Gamma)$, and let $u$ be the canonical fundamental representation of ${QAut}_{Ban}(\Gamma)$. Suppose there exists a vertex $p \in V(\Gamma)$ and  $ k \in \mathbb{N}$ such that $ d(w, p) = k$, and  $ \deg(q) \neq \deg(p)$ for all $ q \in C_k(v)$, where	
	$$	C_k(v) := \left\{ r \in V(\Gamma) \mid d(v, r) = k \right\}.$$ 
	Then it follows that $ u_{wv} = 0$.	
\end{prop}
 
 \begin{prop}{[Theorem 3.5, \cite{SS_Advances}]} \label{Ban=Bic}
 	If $\Gamma$ is  a finite graph that does not contain any quadrangle, then $$(QAut_{Ban}(\Gamma),u) \approx (QAut_{Bic}(\Gamma),v),$$ where $u$ and $v$ denote the canonical fundamental representation of $QAut_{Ban}(\Gamma)$ and $QAut_{Bic}(\Gamma)$ respectively.	
 \end{prop}

\begin{prop} \label{Ban=Ban^c}
	For any finite graph $\Gamma$, let $u$ and $u^c$ denote the canonical fundamental representation of $QAut_{Ban}(\Gamma)$ and $QAut_{Ban}(\Gamma^c)$, respectively. Then $$(QAut_{Ban}(\Gamma),u) \approx (QAut_{Ban}(\Gamma^c),u^c).$$
\end{prop}
We briefly recall the proof of Proposition \ref{Ban=Ban^c}. Since $v$ satisfies Relation \eqref{row_sum_1}, it follows that $v J_n= J_n v$, where $n=|V(\Gamma)|=|V(\Gamma^c)|$. Noting that $A_{\Gamma^c}=J_n-I_n-A_{\Gamma}$,  one can conclude that the relation $A_{\Gamma^c}v=vA_{\Gamma^c}$ is equivalent to $A_{\Gamma}v=vA_{\Gamma}$.

\begin{rem}   \label{Rem_Bic_qsubgrp_Ban^c}
	It is worth emphasizing two key facts related to the quantum automorphism groups of graphs in the sense of Bichon.
\begin{enumerate}
	\item[(a)] An analogue of Proposition \ref{Ban=Ban^c} does not hold in general in the sense of Bichon. For example, $QAut_{Bic}(K_n)$ is commutative, but $QAut_{Bic}(K_n^c)$ is non-commutative for $n \geq 4$. \vspace{0.2cm}
	\item[(b)] Let $(v_{ij})_{1 \leq i,j \leq n}$ denote the canonical generators of $QAut_{Bic}(\Gamma)$. The arguments used in the  proof of Proposition \ref{Ban=Ban^c} ensure that $v_{ij}$'s satisfy the defining relations \eqref{proj} - \eqref{comm_with_adj} of $QAut_{Ban}(\Gamma^c)$, together with the additional commutation relation: $v_{ij}v_{kl}=v_{kl}v_{ij}$ whenever $(i,k) \in E(\Gamma)$ and $(j,l) \in E(\Gamma)$. In other words, $(QAut_{Bic}(\Gamma),v)$ is a quantum subgroup of $(QAut_{Ban}(\Gamma^c),u^c)$.\\
\end{enumerate}
\end{rem}

\noindent The following lemma, which provides a sufficient condition for commutativity of Bichon's quantum automorphism group, can be found as Lemma 3.4 and Lemma 3.5 of \cite{Web}. 
\begin{lem}  \label{Ban=Bic_implies_complement_comm}
Let $\Gamma$ be a finite graph.
\begin{enumerate}
	\item[(i)] If $QAut_{Bic}(\Gamma)$ is a quantum subgroup of $QAut_{Bic}(\Gamma^c)$, then $QAut_{Bic}(\Gamma)$ is commutative.
	\item[(ii)] If $(QAut_{Ban}(\Gamma),u) \approx (QAut_{Bic}(\Gamma),v)$, then $QAut_{Bic}(\Gamma^c)$ is commutative.
\end{enumerate}	
\end{lem}

We now recall a sufficient condition for the existence of quantum symmetry in a graph, as established in \cite{SSthesis}. Before that, we introduce two definitions.
\begin{Def} \label{Def_edgefree_disjoint}
Let $\Gamma$ be a finite graph. Define the {\bf support} of an automorphism $\sigma \in Aut(\Gamma)$ by $$supp(\sigma):= \{v \in V(\Gamma): \sigma(v) \neq v \}.$$
\begin{enumerate}
	\item[(i)] Two automorphisms $\sigma, \tau \in Aut(\Gamma)$ are said to be {\bf disjoint} if $supp(\sigma) \cap supp(\tau) = \emptyset$; that is, if $\sigma(i) \neq i$ then $\tau(i)=i$, and vice versa, for all $i \in V(\Gamma)$.
	\item[(ii)] Two automorphisms $\sigma, \tau \in Aut(\Gamma)$ are said to be {\bf edge-free disjoint} if $supp(\sigma) \cap supp(\tau) = \emptyset$, and there do not exist any edges in $E(\Gamma)$ joining a vertex in $supp(\sigma)$ and a vertex in $supp(\tau)$.
\end{enumerate}
\end{Def}
\begin{ex} \label{example_K4_C4}
	Consider the complete graph $K_4$ on four vertices and the 4-cycle $C_4$, whose vertices are labeled $\{1, 2, 3, 4 \}$, as shown in Figure \ref{Fig_edge free disjoint aut}. Both the graphs contain disjoint automorphisms, for example: $\sigma=(1 ~ 3)$ and $\tau=(2 ~ 4)$. However, these do not form edge-free disjoint automorphisms because, for instance, the vertices 1 and 2 are adjacent. In fact, no edge-free disjoint automorphisms exist for either graph. On the other hand, in their complement graphs, namely $K_4^c$ and $C_4^c ~ (\cong 2K_2$), the automorphisms $\sigma=(1 ~ 3)$ and $\tau=(2 ~ 4)$ are edge-free disjoint because 1 is not adjacent to either 2 or 4, and 3 is also not adjacent to either 2 or 4. 
		\begin{figure}[htbp]
		\centering
		\begin{subfigure}{0.3\textwidth}
			\centering
			\begin{tikzpicture}[scale=1, every node/.style={circle, draw, fill=white, inner sep=1.5pt}]
				
				\node[fill = black, label=above:1] (A) at (-1,1) {};
				\node[fill = black, label=above:2] (B) at (1,1) {};
				\node[fill = black, label=below:3] (C) at (1,-1) {};
				\node[fill = black, label=below:4] (D) at (-1,-1) {};
				
				\node[fill = black, label=above:1] (E) at (2,1) {};
				\node[fill = black, label=above:2] (F) at (4,1) {};
				\node[fill = black, label=below:3] (G) at (4,-1) {};
				\node[fill = black, label=below:4] (H) at (2,-1) {};
				
				\draw (A) -- (B);
				\draw (B) -- (C);
				\draw (C) -- (D);
				\draw (D) -- (A);
				\draw (A) -- (C);
				\draw (B) -- (D);
			\end{tikzpicture}
			\caption{$K_4; K_{4}^c$}  
		\end{subfigure}	 \hspace{3cm}
		\begin{subfigure}{0.3\textwidth}	
			\centering	
			\begin{tikzpicture}[scale=1, every node/.style={circle, draw, fill=white, inner sep=1.5pt}]
				
				\node[fill = black, label=above:1] (A) at (-1,1) {};
				\node[fill = black, label=above:2] (B) at (1,1) {};
				\node[fill = black, label=below:3] (C) at (1,-1) {};
				\node[fill = black, label=below:4] (D) at (-1,-1) {};
				
				\node[fill = black, label=above:1] (E) at (2,1) {};
				\node[fill = black, label=above:2] (F) at (4,1) {};
				\node[fill = black, label=below:3] (G) at (4,-1) {};
				\node[fill = black, label=below:4] (H) at (2,-1) {};
				
				\draw (A) -- (B);
				\draw (B) -- (C);
				\draw (C) -- (D);
				\draw (D) -- (A);
			
			     \draw (E) -- (G);
			     \draw (F) -- (H);
				
			\end{tikzpicture}
			\caption{$C_4; C_4^c$}  \label{}	
		\end{subfigure}	
		\caption{Examples of disjoint automorphisms and edge-free disjoint automorphisms} \label{Fig_edge free disjoint aut}
	\end{figure}
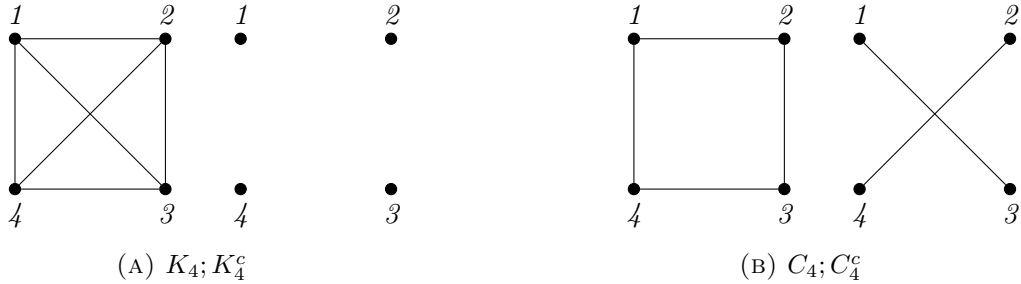
		
\end{ex}

\noindent S. Schmidt provides a sufficient condition for Banica’s quantum automorphism group to be non-commutative \cite{SS_folded cube}, stated as follows:
\begin{thm}{[Theorem 2.2, \cite{SS_folded cube}]} \label{QAut_Ban_noncomm}
Let $\Gamma$ be a finite graph. If there exist two non-trivial, disjoint automorphisms $\sigma, \tau \in Aut(\Gamma)$, then $\Gamma$ has quantum symmetry, i.e. $QAut_{Ban}(\Gamma)$ is non-commutative.
\end{thm}	

An interesting application of the aforementioned theorem in probabilistic aspects arises in the context of quantum automorphism groups of trees \cite{Almosttree}. For further details on quantum automorphism groups of trees, we refer the reader to \cite{Qauttrees}, \cite{Meu}, among others. Before proceeding, we recall the following definition from \cite{Almosttree}.

\begin{Def}
	Let $\Gamma$ be a finite graph. A triple $(v_1, v_2, w)$ of three vertices $v_1, v_2, w \in V(\Gamma)$ is called a cherry if 
	\begin{itemize}
		\item $v_1, v_2$ and  $w$ are pairwise distinct,
		\item  $ (v_1, w), (v_2, w) \in  E(\Gamma)$,
		\item $deg(v_1)=deg(v_2)=1$,
		\item $deg(w)=3$.
		\end{itemize}
\end{Def}
\noindent For the proof of the following proposition, we refer the reader to \cite{Almosttree}.
\begin{thm}   \label{2cherries}
	Almost all trees contain at least two cherries, in the sense that $$lim_{n \to \infty} ~\mathbb{P}[\mathcal{C}_n \geq 2] = 1, $$
	where $\mathcal{C}_n$ denotes the number of cherries in a tree that is drawn uniformly at random from the set of all trees with $n$ vertices.
\end{thm}
\noindent A direct consequence of Theorems \ref{2cherries} and \ref{QAut_Ban_noncomm} is the following:
\begin{thm}{[Theorem 4.2, \cite{Almosttree}]} \label{almost_all_tree}
	Almost all trees have quantum symmetry.\\
\end{thm}

We recall important results regarding the quantum automorphism group of the disjoint union of connected graphs. In the case where all connected components are isomorphic, the following theorem describes the structure of the quantum automorphism group, as in [Theorem 7.5 and Corollary 7.6, \cite{Banver}] and [Theorem 4.2, \cite{Bichonwreath}].
\begin{thm} \label{disjoint_union_k Gamma}
	Let $\Gamma$ be a finite connected graph. Let $u$ denote the canonical fundamental representation of $QAut_{Ban}(k \Gamma)$ (resp. $QAut_{Bic}(k \Gamma)$), and let $x$ denote the canonical fundamental representation of $QAut_{Ban}(\Gamma)$ (resp. $QAut_{Bic}(\Gamma)$).
	Then
	\begin{align*}
	 (QAut_{Ban}(k \Gamma), u) \approx (QAut_{Ban}(\Gamma) \wr_{*} S_k^+, x \wr_{*} t)\\
	\text{\ \  and \ \ \ \ }  (QAut_{Bic}(k \Gamma), u) \approx (QAut_{Bic}(\Gamma) \wr_{*} S_k^+, x \wr_{*} t),
	\end{align*}
	where $(x \wr_{*} t)_{i \alpha, j \beta}:=u_{ij}t_{\alpha \beta}$ and $t:=(t_{\alpha \beta})_{k \times k}$ is the canonical fundamental representation of $S_k^+$.
	For the graph $k\Gamma$, the notation $i\alpha$ refers to the $i^{\text{th}}$ labeled vertex in the $\alpha^{\text{th}}$ component of $k\Gamma$.\\
\end{thm}

\noindent To address the case of distinct connected components, we first recall the following definition.
\begin{Def}
	Two graphs $\Gamma_1$ and $\Gamma_2$ are said to be quantum isomorphic if there exists a $C^*$-algebra $\mathcal{A}$ with generators $u_{ij}$'s (where $i \in V(\Gamma_1)$ and $j \in V(\Gamma_2)$) satisfying the relations
	\begin{align*}
	& u_{ij}=u_{ij}^{*}=u_{ij}^{2}, \\
	& \sum_{i \in V(\Gamma_1)}u_{ij}=1=\sum_{j \in V(\Gamma_2)}u_{ij},\\
	& u A_{\Gamma_1}=A_{\Gamma_2} u.
	\end{align*}
\end{Def}
\noindent It has been shown in Theorem 4.5 of \cite{Lupini} that finite connected graphs $\Gamma_1$ and $\Gamma_2$ are quantum isomorphic if and only if there exist
$i \in V(\Gamma_1)$ and $j \in V(\Gamma_2)$ that are in the same quantum orbit (see Definitions 3.1, 3.3 of \cite{Lupini}) of $\Gamma_1 \sqcup \Gamma_2$. Using this result, it has been shown in Corollary 7.1.4 of \cite{SSthesis} that if finite connected graphs $\Gamma_1$ and $\Gamma_2$ are quantum non-isomorphic, then $QAut_{Ban}(\Gamma_1 \sqcup \Gamma_2) \cong QAut_{Ban}(\Gamma_1) * QAut_{Ban}(\Gamma_2)$. The following provides an extended version of the aforementioned results for any finite family of mutually quantum non-isomorphic graphs, which will be used later.

\begin{lem}{[Lemma 6.3, \cite{Qauttrees}]} \label{quantum_non_isomorphic 1}
	Let $\Gamma$ be a finite graph with connected components $\{\Gamma_{i}\}_{i=1}^{m}$, and let $u:=(u_{ij})_{1 \leq i,j \leq n}$ be the fundamental representation of ${QAut}_{Ban}(\Gamma)$.
	If $u_{pq} \neq 0$ for some $p \in V(\Gamma_i)$ and $q \in V(\Gamma_j)$ for $i \neq j$, then $\Gamma_{i}$ and $\Gamma_{j}$ are quantum isomorphic.
\end{lem}

\begin{lem}{[Lemma 6.4, \cite{Qauttrees}]} \label{quantum_non_isomorphic 2}
		Let $\{\Gamma_{i}\}_{i=1}^{m}$ be some finite graphs such that each
		connected component of $\Gamma_{i}$ is quantum non-isomorphic to a connected component of $\Gamma_{j}$ for all $i \neq j$ .
		Then, $$ QAut_{Ban}(\sqcup_{i=1}^{m} \Gamma_{i}) \cong *_{i=1}^{m} QAut_{Ban}(\Gamma_{i}). $$
\end{lem}  

\begin{rem} \label{quantum_non_isomrphic_Bic}
Since $QAut_{Bic}(\Gamma)$ is a quantum subgroup of $QAut_{Ban}(\Gamma)$ with respect to the canonical fundamental representations, it is straightforward to verify that under the hypothesis of Lemma \ref{quantum_non_isomorphic 2}, we also have
$$ QAut_{Bic}(\sqcup_{i=1}^{m} \Gamma_{i}) \cong *_{i=1}^{m} QAut_{Bic}(\Gamma_{i}). $$\\
\end{rem}
 
\section{Non-commutativity on Quantum Automorphism Group of Bichon} \label{Section_1}
    This section focuses on results concerning the non-commutativity of quantum automorphism groups of graphs in the sense of Bichon.
	The existence of `non-trivial disjoint automorphisms', as mentioned in Theorem \ref{QAut_Ban_noncomm}, is not sufficient to conclude that $QAut_{Bic}(\Gamma)$ is non-commutative; for instance, consider the complete graph with $n$ vertices $K_n$ (for $n \geq 4$). However, we will provide a non-commutativity criterion for the quantum automorphism group of a graph (in the sense of Bichon), analogous to Theorem 3.1.2 of \cite{SSthesis}.
	
	\begin{prop} \label{Bic_noncomm}
		Let $\Gamma$ be a finite graph. If there exist two non-trivial, edge-free disjoint automorphisms in $\Gamma$, then $QAut_{Bic}(\Gamma)$ is non-commutative and  infinite-dimensional.
	\end{prop}
	\begin{proof}
		Let $\sigma$ and $\tau$ be two non-trivial, edge-free disjoint automorphisms, and let $\mathcal{A}:=C^*\{p,q ~|~ p^2=p=p^*, q^2=q=q^*\}$. Define a map $ \phi: QAut_{Bic}(\Gamma)\to \mathcal{A}$ on generators $u_{ij}$'s by $u \mapsto u':= \sigma \otimes p + \tau \otimes q + I_{|V|} \otimes (1-p-q)$, i.e. $u_{ij} \mapsto u_{ij}'$, where $$u_{ij}'= \begin{cases}
		p & if ~~ j=\sigma(i) \text{ and } i,j \in supp(\sigma) \\
		1-p & if ~~ j=i \text{ and } i,j \in supp(\sigma) \\
		q & if ~~ j=\sigma(i) \text{ and } i,j \in supp(\tau)\\
		1-q & if ~~ j=i \text{ and } i,j \in supp(\tau)\\
		\delta_{ij} & if ~~ i,j \notin supp(\sigma) \sqcup supp(\tau)\\
		0 & otherwise.	
		\end{cases} $$
		The arguments presented in Theorem 3.1.2 and Remark 3.1.4 of \cite{SSthesis} ensure that $\{u_{ij}': i,j \in V(\Gamma)\}$ satisfy the defining relations \eqref{proj} - \eqref{comm_with_adj} of $QAut_{Bic}(\Gamma)$. Moreover, it is evident that  $u_{ij}'u_{kl}' \neq u_{kl}'u_{ij}'$ only if $i,j \in supp(\sigma)$ and $k,l \in supp(\tau)$. Finally, since there are no edges connecting any vertex in $supp(\sigma)$ and any vertex in $supp(\tau)$, $\{u_{ij}': i,j \in V(\Gamma)\}$ also satisfy Relation \eqref{Bic_comm_relation} of  $QAut_{Bic}(\Gamma)$. Therefore, the map $\phi$ defines a surjective $C^*$-homomorphism that maps $u_{ij} \mapsto u_{ij}'$. Now, the fact that $\mathcal{A}$ is non-commutative and infinite-dimensional implies that $QAut_{Bic}(\Gamma)$ is as well. \\
	\end{proof}

It is now well known that there exist graphs $\Gamma$ which do not admit any disjoint automorphisms, yet $QAut_{Ban}(\Gamma)$ is non-commutative \cite{asymm, solutiongroup, SSthesis}.
However, at the moment, we do not know any example of a graph $\Gamma$ such that $QAut_{Bic}(\Gamma)$ is non-commutative while $\Gamma$ does not admit any non-trivial edge-free disjoint automorphisms. Although we believe that such an example may exist, the following proposition ensures that no such example exists within the family of forests.
\begin{prop} \label{Forest_Bic}
	Let $F$ be a finite forest. If $QAut_{Bic}(F)$ is non-commutative, then $F$ admits at least two non-trivial, edge-free disjoint automorphisms.
\end{prop}

\noindent Before proceeding to the proof, we introduce some notation and recall a few facts about the structure of a finite tree. It is known that the {\bf center} of a tree $T$, denoted by $Z(T)$, contains either exactly one vertex $\{v_0\}$ or exactly two adjacent vertices $\{v_0, v_1\}$.\\
If $Z(T)=\{v_0\}$, we define $G_i := \left\{ v \in V(T) \mid d(v, v_0) = i \right\}$ (for $i\geq 0$), called the {\bf $i^{th}$ generation} of $T$ (see Figure \ref{Tree_center 1}).\\
If $Z(T)=\{v_0, v_1\}$, we define $G'_0 := \{v_0, v_1\}$ and $
G'_i := \left\{ v \in V \mid \min\left( d(v_0, v), d(v_1, v) \right) = i \right\}$ (for $i \geq 1$), called {\bf $i^{th}$ generation} of $T$ (see Figure \ref{Tree_ center 2}).\\
Since $T$ is finite, for each case there are only finitely many generations, i.e. the index $i$ is bounded.\\
It can be shown that every automorphism of a tree preserves its generations; that is,  for all $i \geq 0$, $\sigma(G_i) \subseteq G_i$ for all $\sigma \in Aut(T)$ if $Z(T)=\{v_0\}$ (respectively, $\sigma(G_i') \subseteq G_i'$ for all $\sigma \in Aut(T)$ if $Z(T)=\{v_0, v_1\}$). This fact will be used later in the proof.
\vspace{.2cm}

	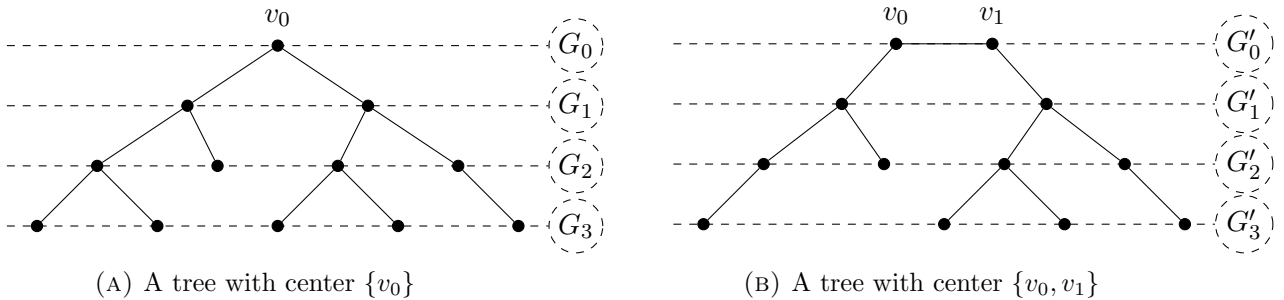
\begin{figure}[htbp]
	\centering
	\begin{subfigure}{0.4\textwidth}
		\centering
		\begin{tikzpicture}[scale=0.8, every node/.style={circle, draw, fill=white, inner sep=1.5pt}]
		
		\node[fill = black,  label=above: $v_0$] (A) at (0,4) {};
		
		\node[fill = black] (B) at (-1.5,3) {};
		\node[fill = black] (C) at (1.5,3) {};
		
		\node[fill = black] (D) at (-3,2) {};
		\node[fill = black] (E) at (-1,2) {};
		\node[fill = black] (F) at (1,2) {};
		\node[fill = black] (G) at (3,2) {};
		
		\node[fill = black] (H) at (-4,1) {};
		\node[fill = black] (I) at (-2,1) {};
		\node[fill = black] (J) at (0,1) {};
		\node[fill = black] (K) at (2,1) {};
		\node[fill = black] (L) at (4,1) {};
		
		\draw (A) -- (B) -- (D) -- (H);
		\draw (B) -- (E);
		\draw (A) -- (C) -- (G) -- (L);
		\draw (C) -- (F) -- (K);
		\draw (F) -- (J);
		\draw (D) -- (I);
		
		\draw[dashed] (-4.5,4) -- (4.5,4) node[right] {\( G_0 \)};
		\draw[dashed] (-4.5,3) -- (4.5,3) node[right] {\( G_1 \)};
		\draw[dashed] (-4.5,2) -- (4.5,2) node[right] {\( G_2 \)};
		\draw[dashed] (-4.5,1) -- (4.5,1) node[right] {\( G_3 \)};
		
	\end{tikzpicture}
		\caption{A tree with center $\{v_0\}$}  \label{Tree_center 1}
	\end{subfigure}	   \hspace{2cm}
	\begin{subfigure}{0.4\textwidth}	
		\centering	
	\begin{tikzpicture}[scale=0.8, every node/.style={circle, draw, fill=white, inner sep=1.5pt}]
		
		\node[fill = black,  label=above: $v_0$] (A') at (-.8,4) {};
		\node[fill = black,  label=above: $v_1$] (A) at (.8,4) {};
		
		\node[fill = black] (B) at (-1.7,3) {};
		\node[fill = black] (C) at (1.7,3) {};
		
		\node[fill = black] (D) at (-3,2) {};
		\node[fill = black] (E) at (-1,2) {};
		\node[fill = black] (F) at (1,2) {};
		\node[fill = black] (G) at (3,2) {};
		
		\node[fill = black] (H) at (-4,1) {};
		\node[fill = black] (J) at (0,1) {};
		\node[fill = black] (K) at (2,1) {};
		\node[fill = black] (L) at (4,1) {};
		
		\draw (A') -- (A);
		\draw (A') -- (B) -- (D) -- (H);
		\draw (B) -- (E);
		\draw (A) -- (C) -- (G) -- (L);
		\draw (C) -- (F) -- (K);
		\draw (F) -- (J);
		
		\draw[dashed] (-4.5,4) -- (4.5,4) node[right] {\( G'_0 \)};
		\draw[dashed] (-4.5,3) -- (4.5,3) node[right] {\( G'_1 \)};
		\draw[dashed] (-4.5,2) -- (4.5,2) node[right] {\( G'_2 \)};
		\draw[dashed] (-4.5,1) -- (4.5,1) node[right] {\( G'_3 \)};
		
	\end{tikzpicture}
		\caption{A tree with center $\{v_0, v_1\}$}  \label{Tree_ center 2}	
	\end{subfigure}	
	\caption{Structure of a tree} \label{Fig_Tree}
\end{figure}

\begin{proof}[Proof of Proposition \ref{Forest_Bic}]
Let $F$ be a finite forest, which is a finite union of finite trees.
Suppose that $QAut_{Bic}(F)$ is non-commutative. Since $QAut_{Bic}(F)$ is a quantum subgroup of $QAut_{Ban}(F)$, it follows that $QAut_{Ban}(F)$ is also non-commutative.
Now, Theorem 6.23 of \cite{Meu} ensures that $F$ admits a pair of non-trivial disjoint automorphisms, say $\sigma$ and $\tau$. We claim that every pair of non-trivial disjoint automorphisms $\sigma, \tau \in {Aut}(F)$ also forms a pair of edge-free disjoint automorphisms.\\
To prove our claim, for the sake of contradiction, let us assume that there exist two vertices $x \in {Supp}(\sigma)$ and $y \in {Supp}(\tau)$ such that $(x,y) \in E(F)$.
Then $x$ and $y$ belong to the same component tree, say $T$, of the forest $F$. Moreover, since $\sigma \in {Aut}(F)$ and $\sigma(y)=y$, we have $(\sigma(x), y) \in E(F)$. Hence, all three distinct vertices $x, y$, and $\sigma(x)$ lie in $V(T)$ with $(x,y), (y,\sigma(x)) \in E(T)$. Therefore, there exists a path $P$ of length two joining $x$ and $\sigma(x)$ through $y$.\\
Since both $x, \sigma(x) \in V(T)$, the connectedness of $T$ ensures that the restriction of $\sigma$ on $T$ is an automorphism of $T$. Recall that every automorphism of a tree preserves its center. Thus, $\sigma(x) \neq x$ and $\tau(y) \neq y$ imply that neither $x$ nor $y$ belongs to $Z(T)=G_0$ whenever $|Z(T)|=1$. Similarly, $\sigma(x) \neq x$ and $\sigma(y) = y$ imply that both $x$ and $y$ cannot simultaneously belong to  $Z(T)=G_0'$ whenever $|Z(T)|=2$.
Hence, $x$ and $y$ must belong to the consecutive generations of $T$. Without loss of generality, assume that $x \in G_i$ and $y \in G_{i+1}$ for some $i \geq 1$ (resp. $x \in G_i'$ and $y \in G_{i+1}'$ for some $i \geq 0$). 
Since $\sigma$ preserves each of the generations of $T$, it follows that $\sigma(x) \in G_i$ (resp. $\sigma(x) \in G_i'$). Therefore, there exists another path  $P_1$ joining $ x $ to $\sigma(x)$ such that $V(P_1) \subseteq \bigcup_{k=0}^{i} G_k$ (resp. $V(P_1) \subseteq \bigcup_{k=0}^{i} G_k' $). The vertex $y \notin V(P_1)$ finally ensures that $P_1$ is distinct from $P$, which yields a contradiction. Thus, $F$ must contain edge-free disjoint automorphisms.
\end{proof}

\noindent Loosely speaking, the class of forests satisfies an analogue of the `Schmidt alternative' (as introduced in \cite{Meu}) in Bichon’s framework.

\begin{rem}	
	In a forest $F$, one has $QAut_{Ban}(F) \cong QAut_{Bic}(F)$. Moreover, in the proof of Proposition \ref{Forest_Bic}, it is shown that every pair of disjoint automorphisms is in fact a pair of edge-free disjoint automorphisms in $F$. This motivates the following general question:\\
    If a finite graph $\Gamma$ admits disjoint automorphisms and $QAut_{Ban}(\Gamma) \cong QAut_{Bic}(\Gamma)$, does every pair of disjoint automorphisms of $\Gamma$ necessarily form an edge-free disjoint pair?
\end{rem}

\noindent Next, we will see a few applications of Proposition \ref{Bic_noncomm}.\\

\noindent Since it is known that $QAut_{Ban}(\Gamma) \cong QAut_{Ban}(\Gamma^c)$, non-commutativity of $QAut_{Ban}(\Gamma)$ automatically implies the non-commutativity of $QAut_{Ban}(\Gamma^c)$, and vice versa. But the same does not hold in the context of Bichon's quantum automorphism group.
Proposition \ref{both_Bic_Bic^c_non-comm} and Example \ref{Example_both _Bic_Bic^c_noncomm} provide an infinite family of graphs for which both $QAut_{Bic}(\Gamma)$ and $QAut_{Bic}(\Gamma^c)$ are non-commutative.
\begin{prop} \label{both_Bic_Bic^c_non-comm}
	Let $\Gamma_1$ be a graph such that $\Gamma_1^c$ contains non-trivial edge-free disjoint automorphisms, and let $\Gamma_2$ be a graph with a non-trivial automorphism group (i.e. $Aut(\Gamma_2) \neq \{id_{V(\Gamma_2)}\}$). Then both $QAut_{Bic}(\Gamma_1 \sqcup \Gamma_2)$ and $QAut_{Bic}((\Gamma_1 \sqcup \Gamma_2)^c)$ are non-commutative.
\end{prop}
\begin{proof}
	Let $\sigma_1$ and $\sigma_2$ be two non-trivial edge-free disjoint automorphisms in $\Gamma_{1}^c$, and let $\tau \in Aut(\Gamma_{2})$ be a non-trivial automorphism. Then each $\sigma_i \in Aut(\Gamma_1)$, and hence the maps $\bar{\sigma_i}, \bar{\tau}: V(\Gamma_{1} \sqcup \Gamma_{2}) \to V(\Gamma_{1} \sqcup \Gamma_{2})$ defined by 
	$$ \bar{\sigma_i}(v) := \begin{cases}
	\sigma_i(v) & v \in V(\Gamma_1)\\
	v & v \in V(\Gamma_{2})
	\end{cases}  \text{ (for $i=1,2$) \ \ \ \ \ \ and \ \ \ \ \ \ }  \bar{\tau}(v) := \begin{cases}
	\tau(v) & v \in V(\Gamma_2) \\
	v & v \in V(\Gamma_{1})
	\end{cases} $$
	are automorphisms in both $(\Gamma_1 \sqcup \Gamma_2)$ and $(\Gamma_1 \sqcup \Gamma_2)^c$. Thus, $\bar{\sigma_1}$ and $\bar{\tau}$  form non-trivial edge-free disjoint automorphisms in $Aut(\Gamma_1 \sqcup \Gamma_2)$, hence by Proposition \ref{Bic_noncomm}, $QAut_{Bic}(\Gamma_1 \sqcup \Gamma_2)$ is non-commutative. \\
	On the other hand, observe that every $v \in supp(\sigma_1) \cup supp(\sigma_2)$ is adjacent to all the vertices $w \in V(\Gamma_2)$ in $(\Gamma_1 \sqcup \Gamma_2)^c$, while there are no edges joining a vertex in $supp(\sigma_1)$ with a vertex belonging to $supp(\sigma_2)$ in the graph $(\Gamma_1 \sqcup \Gamma_2)^c$. This ensures that $\bar{\sigma_1}$ and $\bar{\sigma_2}$ remain non-trivial edge-free disjoint automorphisms in  $(\Gamma_1 \sqcup \Gamma_2)^c$. Thus, again by Proposition \ref{Bic_noncomm}, we conclude that $QAut_{Bic}((\Gamma_1 \sqcup \Gamma_2)^c)$ is non-commutative.
\end{proof}

\begin{ex} \label{both_Bic_Bic^c_comm}    \label{Example_both _Bic_Bic^c_noncomm}
We provide two explicit families of graphs as described in Proposition \ref{both_Bic_Bic^c_non-comm}.\\
(1) Let $\Gamma_2$ be a graph with a non-trivial automorphism group, and consider the graph $\Gamma:= C_4 \sqcup \Gamma_2$, where $V(C_4)$ is labeled  by \{1,2,3,4\} (for example, take $\Gamma_2 =P_1 $; see Figure \ref{Gamma_Gamma^c_both_comm}). Let  $id_{V(\Gamma_2)} \neq \tau \in Aut(\Gamma_2)$; for example, we may take $\tau=(5 ~ 6)$ in case of $\Gamma_{2}=P_2$. Note that $\bar{\sigma_1}=(1 ~ 3)$ and $\bar{\tau}$ are edge-free disjoint automorphisms in $Aut(\Gamma)$. Hence, by Proposition \ref{Bic_noncomm}, $QAut_{Bic}(\Gamma)$ is non-commutative.

	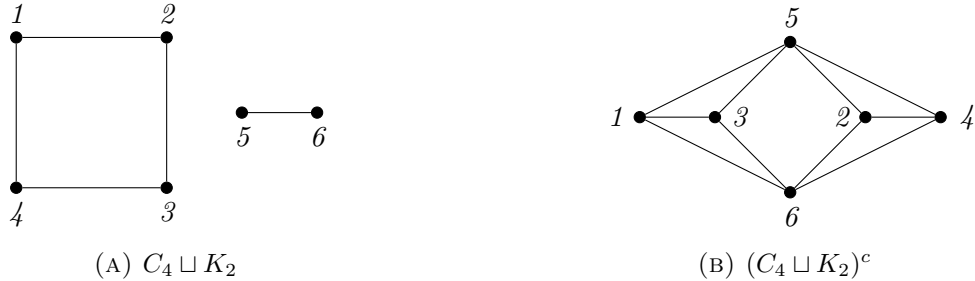
\begin{figure}[htbp]
			\centering
		\begin{subfigure}{0.3\textwidth}
			\centering
			\begin{tikzpicture}[scale=1, every node/.style={circle, draw, fill=white, inner sep=1.5pt}]
			
			\node[fill = black, label=above:1] (A) at (-1,1) {};
			\node[fill = black, label=above:2] (B) at (1,1) {};
			\node[fill = black, label=below:3] (C) at (1,-1) {};
			\node[fill = black, label=below:4] (D) at (-1,-1) {};
			
			\node[fill = black, label=below:5] (E) at (2,0) {};
			\node[fill = black, label=below:6] (F) at (3,0) {};

			\draw (A) -- (B);
			\draw (B) -- (C);
			\draw (C) -- (D);
			\draw (D) -- (A);
			\draw (E) -- (F);
			\end{tikzpicture}
			\caption{$C_4 \sqcup K_2$}  \label{C4_union_Gamma1}
			\end{subfigure}	   \hspace{3cm}
		\begin{subfigure}{0.3\textwidth}	
         \centering	
    	\begin{tikzpicture}[scale=1, every node/.style={circle, draw, fill=white, inner sep=1.5pt}]
		
		\node[fill = black, label=left:1] (A) at (-2,0) {};
		\node[fill = black, label=right:3] (B) at (-1,0) {};
		\node[fill = black, label=left:2] (C) at (1,0) {};
		\node[fill = black, label=right:4] (D) at (2,0) {};
		
		\node[fill = black, label=above:5] (E) at (0,1) {};
		\node[fill = black, label=below:6] (F) at (0,-1) {};

		\draw (A) -- (B);
		\draw (C) -- (D);
		
		\draw (E) -- (A);
		\draw (E) -- (B);
		\draw (E) -- (C);
		\draw (E) -- (D);
		
		\draw (F) -- (A);
		\draw (F) -- (B);
		\draw (F) -- (C);
		\draw (F) -- (D);
		
		\end{tikzpicture}
		\caption{$(C_4 \sqcup K_2)^c$}  \label{C4_union_Gamma1_complement}	
		\end{subfigure}	
	\caption{A graph $\Gamma$ and its complement $\Gamma^c$ with non-commutative $QAut_{Bic}(.)$} \label{Gamma_Gamma^c_both_comm}
	\end{figure}
	\noindent Furthermore, note that $\bar{\sigma_{1}}:=(1 ~ 3)$ and $\bar{\sigma_{2}} :=(2 ~ 4)$ are two non-trivial disjoint edge-free automorphisms in $Aut(\Gamma^c)$. Therefore, again using Proposition \ref{Bic_noncomm}, we have $QAut_{Bic}(\Gamma^c)$ is also non-commutative.\\
	
\noindent (2) It should be noted that since $K_{m,n} \cong (K_m \sqcup K_n)^c $, Proposition \ref{both_Bic_Bic^c_non-comm} directly implies that both $QAut_{Bic}(K_{m,n})$ and $QAut_{Bic}(K_{m,n}^c)$ are non-commutative whenever $m \geq 4$ and $n \geq 2$. Thus, the complete bipartite graphs $K_{m,n}$ (for $m \geq 4$ and $n \geq 2$) provide another infinite family of graphs having both non-commutative $QAut_{Bic}(K_{m,n})$ and $QAut_{Bic}(K_{m,n}^c)$.	
\end{ex}

\noindent The following proposition characterizes the non-commutativity of Bichon’s quantum automorphism groups for complete bipartite graphs $K_{m,n}$.

\begin{prop} \label{K_mn_noncomm}
	$QAut_{Bic}(K_{m,n})$ is non-commutative if and only if $m \geq 4$ or $n \geq 4$.
\end{prop}
\begin{proof}
  Let $M:=\{v_1, v_2, ..., v_m\}$ denote one of the partitions of the vertex set $V(K_{m,n})$.\\
  Without loss of generality, let $m \geq 4$ and $n \geq 1$. Then, $(v_1 ~ v_2)$ and $(v_3 ~ v_4)$ are clearly two non-trivial edge-free disjoint automorphisms in $Aut(K_{m,n})$. Therefore, it follows from Proposition \ref{Bic_noncomm}  that $QAut_{Bic}(K_{m,n})$ is non-commutative.\\
 Conversely, let us assume that both $m, n \leq 3$.\\
   For $(m,n)=(1,1), (1,2)$ or $(1,3)$, clearly $QAut_{Bic}(K_{m,n})$ is commutative.
   By Theorem 3.8 (4) of \cite{Web}, $QAut_{Bic}(K_{2,2})$ is also commutative.\\
   Finally, consider the graphs $K_{2,3}$ and $K_{3,3}$. Observe that their complement graphs $K_{2,3}^c$ and $K_{3,3}^c$ are isomorphic to $K_2 \sqcup K_3$ and $K_3 \sqcup K_3$, respectively, both of which do not contain any quadrangle. By Lemma \ref{Ban=Bic}, it follows that $QAut_{Ban}(K_{2,3}^c) \approx QAut_{Bic}(K_{2,3}^c)$ and $QAut_{Ban}(K_{3,3}^c) \approx QAut_{Bic}(K_{3,3}^c)$. Furthermore, Lemma \ref{Ban=Bic_implies_complement_comm} (ii) ensures that both $QAut_{Bic}(K_{2,3})$ and $QAut_{Bic}(K_{3,3})$ are commutative.  
\end{proof}

\noindent To establish the non-commutativity of $QAut_{Bic}(K_{m,n})$, it suffices to observe that $S_m^+$ and $S_n^+$ are the quantum subgroups of $QAut_{Bic}(K_{m,n})$. However, we present the above proof as an application of Proposition \ref{Bic_noncomm}.\\

Observe that every such graph $\Gamma$ appearing in Example \ref{both_Bic_Bic^c_comm} contains a quadrangle. In fact, Proposition \ref{Ban=Bic} and Lemma \ref{Ban=Bic_implies_complement_comm} ensure that the presence of a quadrangle in $\Gamma$ is necessary for $QAut_{Bic}(\Gamma^c)$ to be non-commutative. Nevertheless, the existence of a quadrangle in a graph $\Gamma$ is not sufficient to conclude the commutativity of $QAut_{Bic}(\Gamma^c)$. However, Proposition \ref{K_mn_noncomm} implies the following:
\begin{rem}
	 For a graph $\Gamma$ belonging to the family of complete bipartite graphs $\{K_{m,n}: m,n \in \mathbb{N}\}$, $QAut_{Bic}(\Gamma)$ is non-commutative if and only if its complement graph $\Gamma^c$ contains a quadrangle.\\
\end{rem}

\noindent Lemma \ref{Ban=Bic_implies_complement_comm} (i) implies that there does not exist any graph $\Gamma$ with non-commutative $QAut_{Bic}(\Gamma)$ such that $QAut_{Bic}(\Gamma)$ is identical to $QAut_{Bic}(\Gamma^c)$ with respect to their canonical fundamental representations. The following example shows that the term ``identical" can not be replaced by ``isomorphic". More precisely, there exists a graph, namely $\Gamma_{sc}$ (see Figure \ref{Gamma_self_comp}), for which $QAut_{Bic}(\Gamma_{sc}) \cong  QAut_{Bic}(\Gamma_{sc}^c)$, even though $QAut_{Bic}(\Gamma_{sc})$ is non-commutative.

\begin{figure}[htbp]
	\centering
	\begin{tikzpicture}[scale=1, every node/.style={circle, draw, fill=white, inner sep=1.5pt}]
	
	\node[fill = black, label=above:1] (A) at (-1,1) {};
	\node[fill = black, label=above:2] (B) at (1,1) {};
	\node[fill = black, label=below:3] (C) at (1,-1) {};
	\node[fill = black, label=below:4] (D) at (-1,-1) {};
	
	\node[fill = black, label=left:8] (E) at (-2,1) {};
	\node[fill = black, label=left:7] (F) at (-2,-1) {};
	\node[fill = black, label=right:5] (G) at (2,1) {};
	\node[fill = black, label=right:6] (H) at (2,-1) {};

	\draw (A) -- (B);
	\draw (B) -- (C);
	\draw (C) -- (D);
	\draw (D) -- (A);	
	\draw (A) -- (C);
	\draw (B) -- (D);

	\draw (A) -- (E);
	\draw (D) -- (F);
	\draw (B) -- (G);
	\draw (C) -- (H);
	\draw (E) -- (D);
	\draw (A) -- (F);
	\draw (C) -- (G);
	\draw (B) -- (H);
	
	\end{tikzpicture}
	\caption{$\Gamma_{sc}$}  \label{Gamma_self_comp}
\end{figure}
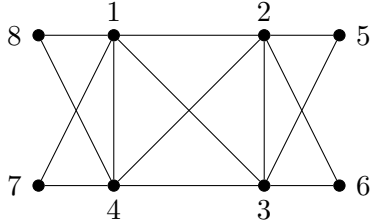

\begin{ex}    \label{example_self-comp}
	Consider the self-complementary graph $\Gamma_{sc}$ with vertex set $V(\Gamma_{sc})=\{1,2,...,8\}$, as shown in Figure \ref{Gamma_self_comp}. Note that $\sigma':=(5~6)$ and $\tau':=(7~8)$ (respectively, $\sigma:=(2~3)$ and $\tau:=(1~4)$) are edge-free disjoint automorphisms in $Aut(\Gamma)$ (respectively, $Aut(\Gamma^c)$). Hence both $QAut_{Bic}(\Gamma_{sc})$ and $QAut_{Bic}(\Gamma_{sc}^c)$ are non-commutative (by Proposition \ref{Bic_noncomm}). Let $\psi: V(\Gamma_{sc}) \to V(\Gamma_{sc}^c)$ be a graph isomorphism (say, $\psi= (1 ~ 5 ~ 2 ~ 8)(3 ~ 7 ~ 4 ~ 6)$), and let $u$ and $v$ denote the canonical fundamental representation of $QAut_{Bic}(\Gamma_{sc})$ and $QAut_{Bic}(\Gamma_{sc}^c)$ respectively. Then $\widetilde{\psi} : QAut_{Bic}(\Gamma_{sc}) \to QAut_{Bic}(\Gamma_{sc}^c) $ defined by $u_{ij} \mapsto v_{\psi(i)\psi(j)}$ is a CQG-isomorphism. However, $(QAut_{Bic}(\Gamma_{sc}),u)$ is not identical to $(QAut_{Bic}(\Gamma_{sc}^c),v)$. \\
	In general, for any self-complementary graph $\Gamma$ containing edge-free disjoint automorphisms, we have  $QAut_{Bic}(\Gamma) \cong  QAut_{Bic}(\Gamma^c)$, even though $QAut_{Bic}(\Gamma)$ is non-commutative. 
\end{ex}
\vspace{0.3cm}

Next, we see an application involving line graphs. It is known that $QAut_{Bic}(\Gamma)$ and  $QAut_{Bic}(L(\Gamma))$ are not equal in general; in fact, one of them can be commutative while the other is non-commutative. For instance, consider the family of trees $K_{1,n}$ for $n \geq 4$, where each $QAut_{Bic}(K_{1,n})$ is non-commutative but $QAut_{Bic}(L(K_{1,n}))=QAut_{Bic}(K_n)$ is commutative. The following proposition ensures that there exist infinitely many graphs $\Gamma$ for which both $QAut_{Bic}(\Gamma)$ and $QAut_{Bic}(L(\Gamma))$ are non-commutative. 
      
\begin{prop} \label{L(T)_noncomm}
	If $\Gamma$ is a graph with at least two cherries, then $QAut_{Bic}(L(\Gamma))$ is non-commutative.   
\end{prop}
\begin{proof}
	Let $(v, u_1, u_2)$ and $(w, u_3, u_4)$ be two cherries. Then there exist $v', w' \in V(\Gamma)$ such that $(v,v'), (w, w') \in E(\Gamma)$.  Let $e_1, e_2, e_3, e_4, e_5, e_6$ be the labeling of the vertices in $V(L(\Gamma))$ corresponding to the edges $(u_1, v), (u_2,v), (v,v'), (u_3,w), (u_4,w), (w,w') \in E(\Gamma)$ respectively (see Figure \ref{graph_eith_2 cherries} for an illustration). Now, note that $e_1, e_2, e_3$ and $e_4, e_5, e_6$ form two  3-cycles with $deg(e_i)=2$ for $i=1,2,4,5$. Since $\sigma= (e_1 ~ e_2)$ and $(e_4 ~ e_5)$ are two non-trivial edge-free disjoint automorphisms in $Aut(L(\Gamma))$, by Proposition \ref{Bic_noncomm}, $QAut_{Bic}(L(\Gamma))$ is non-commutative.
	\begin{figure}[htbp]
		\centering
		\begin{subfigure}{0.3\textwidth}
			\centering
			\begin{tikzpicture}[scale=1, every node/.style={circle, draw, fill=white, inner sep=1.5pt}]
			
			\node[fill = black, label=above: $v$] (A) at (-1,4) {};
			\node[fill = black, label=above:$w$] (A') at (1,4) {};
			
			\node[fill = black, label=below:$u_1$] (B) at (-1.7,3) {};
			\node[fill = black, label=below:$u_2$] (B') at (-0.3,3) {};
			
			\node[fill = black, label=below:$u_3$] (C) at (0.3,3) {};
			\node[fill = black, label=below:$u_4$] (C') at (1.7,3) {};

			\draw (A') -- (A);
			\draw (A) -- (B);
			\draw (A) -- (B');
			\draw (A') -- (C);
			\draw (A') -- (C');
			\end{tikzpicture}
			\caption{$\Gamma_{cherry}$}  \label{Fig_2 cherries}
		\end{subfigure}    \hspace{3cm}
		\begin{subfigure}{0.3\textwidth}
			\centering
			\begin{tikzpicture}[scale=1, every node/.style={circle, draw, fill=white, inner sep=1.5pt}]
			
			\node[fill = black, label=above:$e_3$] (A) at (0,0) {};
			\node[fill = black, label=above:$e_2$] (B) at (-1,0.5) {};
			\node[fill = black, label=below:$e_1$] (C) at (-1,-0.5) {};
			
			\node[fill = black, label=above:$e_5$] (B') at (1,0.5) {};
			\node[fill = black, label=below:$e_4$] (C') at (1,-0.5) {};
			
			\draw (A) -- (B);
			\draw (A) -- (C);
			\draw (A) -- (B');
			\draw (A) -- (C');
			\draw (B) -- (C);
			\draw (B') -- (C');
			\end{tikzpicture}	
			\caption{L($\Gamma_{cherry}$)}  \label{linegraph_2 cherries}
		\end{subfigure}	
		\caption{A graph with two cherries and its line graph. (As an illustration of Proposition \ref{L(T)_noncomm}, in Figure \ref{Fig_2 cherries} we have $w = v'$, $w' = v$, and in Figure \ref{linegraph_2 cherries}, $e_3 = e_6$.)}
		\label{graph_eith_2 cherries}
	\end{figure}
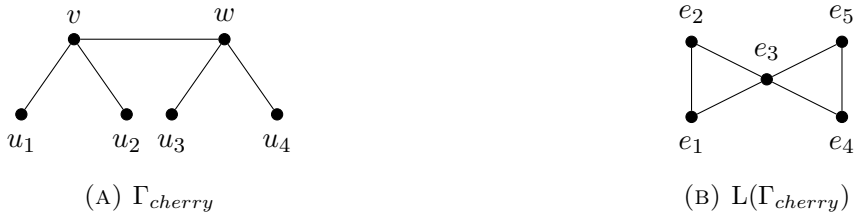 
\end{proof}

\noindent A consequence of Theorems \ref{2cherries}, \ref{almost_all_tree} and Proposition \ref{L(T)_noncomm} is the following:
\begin{cor}
	For almost all trees $T$, both $QAut_{Bic}(T)$ and $QAut_{Bic}(L(T))$ are non-commutative.
\end{cor}

\noindent  It is worth noting that the converse of Proposition \ref{L(T)_noncomm} is not true. For example, consider the tree $T_0$ as shown in Figure \ref{T_0}. Observe that  $T_0$ does not contain any cherry, but both $T_0$ and $L(T_0)$ admit edge-free disjoint automorphisms. Hence, by Proposition \ref{Bic_noncomm}, $QAut_{Bic}(L(T_0))$ is non-commutative.\\

\begin{figure}[htbp]
	\centering
	\begin{subfigure}{0.3\textwidth}
		\centering
		\begin{tikzpicture}[scale=1, every node/.style={circle, draw, fill=white, inner sep=1.5pt}]
		
		\node[fill = black] (A) at (-1,4) {};
		\node[fill = black] (A') at (1,4) {};
		
		\node[fill = black] (B) at (-1.5,3) {};
		\node[fill = black] (B') at (-0.5,3) {};
		
		\node[fill = black] (C) at (0.5,3) {};
		\node[fill = black] (C') at (1.5,3) {};
		
		\node[fill = black] (D) at (-1.7,2) {};
		\node[fill = black] (D') at (-1.3,2) {};
		\node[fill = black] (E) at (-0.7,2) {};
		\node[fill = black] (E') at (-0.3,2) {};
		\node[fill = black] (F) at (0.3,2) {};
		\node[fill = black] (F') at (0.7,2) {};
		\node[fill = black] (G) at (1.3,2) {};
		\node[fill = black] (G') at (1.7,2) {};
		
		\node[fill = black] (I) at (-1.3,1) {};
		\node[fill = black] (J) at (-0.3,1) {};
		\node[fill = black] (K) at (0.7,1) {};
		\node[fill = black] (L) at (1.7,1) {};
		
		\draw (A') -- (A);
		\draw (A) -- (B);
		\draw (A) -- (B');
		\draw (A') -- (C);
		\draw (A') -- (C');
		
		\draw (B) -- (D);
		\draw (B) -- (D');
		\draw (B') -- (E);
		\draw (B') -- (E');
		\draw (C) -- (F);
		\draw (C) -- (F');
		\draw (C') -- (G);
		\draw (C') -- (G');
		
		\draw (D') -- (I);
		\draw (E') -- (J);
		\draw (F') -- (K);
		\draw (G') -- (L);

		\end{tikzpicture}
		\caption{$T_0$} 
	\end{subfigure}  \hspace{3cm}
	\begin{subfigure}{0.3\textwidth}
		\centering
		\begin{tikzpicture}[scale=1, every node/.style={circle, draw, fill=white, inner sep=1.5pt}]
		
		\node[fill = black] (A) at (0,0) {};
		
		\node[fill = black] (B) at (-1,0.5) {};
		\node[fill = black] (C) at (-1,-0.5) {};
		
		\node[fill = black] (B') at (1,0.5) {};
		\node[fill = black] (C') at (1,-0.5) {};
		
	\node[fill = black] (D) at (2,0.5) {};
	\node[fill = black] (D') at (1,1.5) {};
	
	\node[fill = black] (E) at (2,-0.5) {};
	\node[fill = black] (E') at (1,-1.5) {};
	
	\node[fill = black] (F) at (-2,0.5) {};
	\node[fill = black] (F') at (-1,1.5) {};
	
	\node[fill = black] (G) at (-2,-0.5) {};
	\node[fill = black] (G') at (-1,-1.5) {};
	
	\node[fill = black] (H) at (3,0.5) {};
	\node[fill = black] (I) at (3,-0.5) {};
	\node[fill = black] (J) at (-3, 0.5) {};
	\node[fill = black] (K) at (-3,-0.5) {};

		\draw (A) -- (B);
		\draw (A) -- (C);
		\draw (A) -- (B');
		\draw (A) -- (C');
		\draw (B) -- (C);
		\draw (B') -- (C');
		
		\draw (B') -- (D);
		\draw (B') -- (D');
		\draw (D) -- (D');
		
		\draw (C') -- (E);
		\draw (C') -- (E');
		\draw (E) -- (E');
		
		\draw (B) -- (F);
		\draw (B) -- (F');
		\draw (F) -- (F');
		
		\draw (C) -- (G);
		\draw (B) -- (G');
		\draw (G) -- (G');
	
		\draw (D) -- (H);
		\draw (E) -- (I);
		\draw (G) -- (K);
		\draw (F) -- (J);	
		
		\end{tikzpicture}	
		\caption{L($T_0$)} 
	\end{subfigure}	
	\caption{A graph without cherry and its line graph} \label{T_0}
\end{figure}
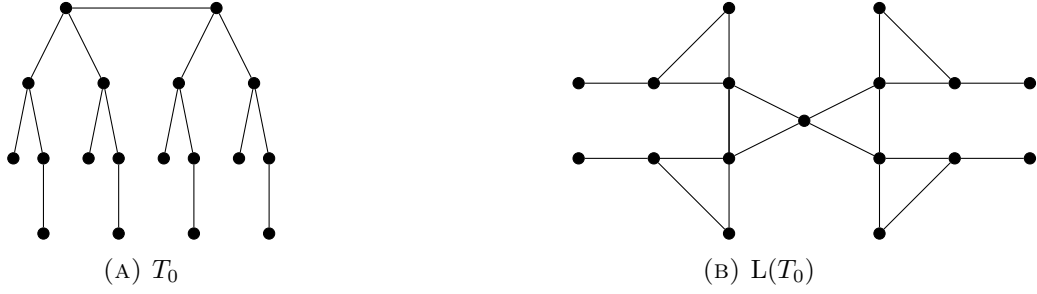 

\noindent However, in general, the existence of non-trivial edge-free disjoint automorphisms in a graph $\Gamma$ does not guarantee that its line graph $L(\Gamma)$ also admits edge-free disjoint automorphisms. For instance, consider the graph $\Gamma:=K_{1,n}$ (for $n \geq 4$), which contains non-trivial edge-free disjoint automorphisms, whereas its line graph $L(K_{1,n}) \cong K_n$ contains no edge-free disjoint automorphism. \\ 

We provide a family of graphs, constructed via the corona product of two graphs (see Definition \ref{Def_corona_prod}), for which Bichon's quantum automorphism group is non-commutative.

\begin{prop} \label{corona_prod}
    Let $\Gamma_1$ and $\Gamma_2$ be  graphs with $|V(\Gamma_1)| \geq 2$. If $Aut(\Gamma_2)$ is non-trivial, then $QAut_{Bic}(\Gamma_1 \odot \Gamma_2)$ is non-commutative. 
\end{prop}
\begin{proof}
	Let $\sigma \in Aut(\Gamma_2)$ be a non-trivial automorphism. For each $v \in V(\Gamma_{1})$, define a map $\bar{\sigma}_{v}: V(\Gamma_1 \odot \Gamma_{2}) \to V(\Gamma_1 \odot \Gamma_{2})$ by $$\bar{\sigma}_{v}(x)=x \text{ for all } x \in V(\Gamma_1)$$ and $$\bar{\sigma}_{v}(w,w')=\begin{cases} (\sigma(w), v) & if ~ (w,w') \in V(\Gamma_{2}) \times V(\Gamma_{1}) \text{ and } w'=v\\ (w,w') & if ~ (w,w') \in V(\Gamma_{2}) \times V(\Gamma_{1}) \text{ and } w' \neq v
	\end{cases}.$$ By the construction of the corona product, $\bar{\sigma}_{v} \in Aut(\Gamma_1 \odot \Gamma_2)$ for all $v \in V(\Gamma_1)$. Since  $|V(\Gamma_1)| \geq 2$, there exist at least two distinct vertices $v, v' \in V(\Gamma_1)$. Now, the structure of the edge set of $\Gamma_1 \odot \Gamma_2$ ensures that $\bar{\sigma}_{v}$ and $\bar{\sigma}_{v'}$ are edge-free disjoint automorphisms in $\Gamma_{1} \odot \Gamma_{2}$. Therefore, it follows from Proposition \ref{Bic_noncomm} that $QAut_{Bic}(\Gamma_1 \odot \Gamma_2)$ is non-commutative.
\end{proof}


It is worth noting that such a minimal hypothesis is insufficient to ensure the  non-commutativity of $QAut_{Bic}(.)$, in general, for other standard graph products like Cartesian product ($\square$), tensor product ($\times$), strong product ($\boxtimes$) and lexicographic product ($\circ$). For example, both the graphs $K_2$ and $K_3$ contain at least 2 vertices with non-trivial automorphism groups; however, $QAut_{Bic}(K_2 \square K_2)$, $QAut_{Bic}(K_2 \boxtimes K_2)$ and $QAut_{Bic}(K_2 \times K_3)$ are all commutative. Moreover, for a (quantum) asymmetric tree T, since $T \circ K_2^c \cong T \sqcup T$, its quantum automorphism group $QAut_{Bic}(T \circ K_2^c) \cong QAut_{Bic}(T) \wr_{*} S_2^+$ is also commutative.\\


The following proposition is the analogue of Proposition 2.2.1 in \cite{SSthesis} for the Bichon's framework, which is sometimes helpful to determine the non-commutativity of $QAut_{Bic}(.)$.
  \begin{prop} \label{Bic_noncomm_prod_graph}
  	Let $\Gamma_1$ and $\Gamma_2$ be finite graphs with $V(\Gamma_1)=n$ and $V(\Gamma_2)=m$. Then there
  	exist surjective *-homomorphisms on the following $C^*$-algebras:
  	\begin{enumerate}
  		\item[(i)] $QAut_{Bic}(\Gamma_1 \square \Gamma_2) \to QAut_{Bic}(\Gamma_1) \otimes QAut_{Bic}(\Gamma_2)$ \ \ via \ \  $w \mapsto u \otimes v $.
  		\item[(ii)] $QAut_{Bic}(\Gamma_1 \times \Gamma_2) \to QAut_{Bic}(\Gamma_1) \otimes QAut_{Bic}(\Gamma_2)$ \ \ via \ \ $w \mapsto u \otimes v  $.
  		\item[(iii)] $QAut_{Bic}(\Gamma_1 \boxtimes \Gamma_2) \to QAut_{Bic}(\Gamma_1) \otimes QAut_{Bic}(\Gamma_2)$ \ \ via \ \ $w \mapsto u \otimes v  $.
  		\item[(iv)] $QAut_{Bic}(\Gamma_1 \circ \Gamma_2) \to QAut_{Bic}(\Gamma_1)$ \ and \  $QAut_{Bic}(\Gamma_1 \circ \Gamma_2) \to QAut_{Bic}(\Gamma_2)$ \ \ via \ \ $$w \mapsto I_n \otimes v  \text{ \hspace{1cm} and \hspace{1cm}} w \mapsto w':= \begin{bmatrix}
  		u & 0 & 0 & \cdots & 0 \\
  		0 & I_n & 0 & \cdots & 0 \\
  		0 & 0 & I_n & \cdots & 0 \\
  		\vdots & \vdots & \vdots & \ddots & 0\\
  		0 & 0 & 0 & \cdots & I_n  
  		\end{bmatrix}_{mn \times mn}$$ respectively.
\end{enumerate}
Here $u$ and $v$ denote the fundamental representations of $QAut_{Bic}(\Gamma_1)$ and $QAut_{Bic}(\Gamma_2)$ respectively, and $w$ denotes the fundamental representation $QAut_{Bic}(.)$ for the corresponding product graphs.
  \end{prop}
\begin{proof}
By Theorem 2.2.1 of \cite{SSthesis}, it follows that the magic unitary matrix $u \otimes v$ commutes with the respective adjacency matrix of the product graphs in (i), (ii) and (iii). We need to verify that $u \otimes v$ satisfies Relation \eqref{Bic_comm_relation} for the product graphs $\Gamma_1 \square \Gamma_2$, $\Gamma_1 \times \Gamma_2$ and $\Gamma_1 \boxtimes \Gamma_2$. To that end, we have to show that if $(i, \alpha)$ is adjacent to $(k, \gamma)$ and $(j, \beta)$ is adjacent to $(l, \delta)$ in the product graph, then $(u_{ij}v_{\alpha \beta})(u_{kl}v_{\gamma \delta})=(u_{kl}v_{\gamma \delta})(u_{ij}v_{\alpha \beta})$, which is equivalent to 
\begin{equation} \label{Eqn_prod_graph}
u_{ij}u_{kl}v_{\alpha \beta}v_{\gamma \delta}=u_{kl}u_{ij}v_{\gamma \delta}v_{\alpha \beta}.
\end{equation} Here, we check Equation \eqref{Eqn_prod_graph} for the Cartesian product case only, as the arguments for the other product cases are almost similar.
Let $((i, \alpha), (k, \gamma))$ and $((j, \beta),(l, \delta))$ belong to $E(\Gamma_1 \square \Gamma_2)$.\\

\noindent $\bullet$ If ($i=k$ and $(\alpha, \gamma) \in E(\Gamma_2)$) and ($j=l$ and $(\beta, \delta) \in E(\Gamma_2)$), then Equation \eqref{Eqn_prod_graph} holds because of the relations $v_{\alpha \beta} v_{\gamma \delta} = v_{\gamma \delta} v_{\alpha \beta}$ and $u_{ij}u_{kl}=u_{ij}=u_{kl}u_{ij}$.\\
\noindent $\bullet$ If  ($(i,k) \in E(\Gamma_1)$ and $\alpha = \gamma$ ) and ($(j,l) \in E(\Gamma_1)$ and $\beta= \delta$), then $u_{ij}u_{kl}=u_{kl}u_{ij}$ and $v_{\alpha \beta}v_{\gamma \delta}=v_{\alpha \beta}=v_{\gamma \delta}v_{\alpha \beta}$, which imply Equation \eqref{Eqn_prod_graph}.\\
\noindent $\bullet$ If ($i=k$ and $(\alpha, \gamma) \in E(\Gamma_2)$) and ($(j,l) \in E(\Gamma_1)$ and $\beta= \delta$), then $u_{ij}u_{kl}=0=u_{kl}u_{ij}$ and hence \eqref{Eqn_prod_graph} holds.\\
\noindent $\bullet$ If ($(i,k) \in E(\Gamma_1)$ and $\alpha = \gamma$ ) and ($j=l$ and $(\beta, \delta) \in E(\Gamma_2)$), then similarly \eqref{Eqn_prod_graph} follows from the fact $v_{\alpha \beta}v_{\gamma \delta}=0=v_{\gamma \delta}v_{\alpha \beta}$.\\

Finally, for (iv), first of all observe that $A_{\Gamma_1 \circ \Gamma_2} = A_{\Gamma_1} \otimes  I_m + J_n \otimes A_{\Gamma_2}$ commutes with $(I_n \otimes v)$, since $A_{\Gamma_2}v=vA_{\Gamma_2}$. To ensure the commutativity relation \eqref{Bic_comm_relation} for the graph $\Gamma_1 \circ \Gamma_2$, consider two edges $((i,\alpha), (k, \gamma)) \text{ and } ((j, \beta), (l, \eta)) \in E(\Gamma_1 \circ \Gamma_2)$. We have to show that $(\delta_{ij}v_{\alpha \beta})(\delta_{kl}v_{\gamma \eta})=(\delta_{kl}v_{\gamma \eta})(\delta_{ij}v_{\alpha \beta})$, which is equivalent to showing that
\begin{equation} \label{Eqn_lexico_prod}
\delta_{ij} \delta_{kl} v_{\alpha \beta} v_{\gamma \eta} = \delta_{kl}\delta_{ij}  v_{\gamma \eta} v_{\alpha \beta}.
\end{equation}
$\delta_{ij} \delta_{kl}=\delta_{kl}\delta_{ij}$ holds trivially for all $i,j,k,l \in V(\Gamma_1)$. By definition of $\Gamma_1 \circ \Gamma_2$, one of the following holds: (a) $(\alpha, \gamma), (\beta, \eta) \in E(\Gamma_2)$, (b) $\alpha = \gamma$, or (c) $\beta = \eta$.  In each case, since $v_{\alpha \beta} v_{\gamma \eta} = v_{\gamma \eta} v_{\alpha \beta}$, Equation \eqref{Eqn_lexico_prod} follows.

\noindent Similarly, it is straightforward to verify that $w'$ commutes with $A_{\Gamma_1 \circ \Gamma_2}$ and satisfies the commutativity relation \eqref{Bic_comm_relation}.
\end{proof}

\noindent It is important to note that the analogue of Proposition 2.2.1 (iv) from \cite{SSthesis} is not true in the framework of Bichon’s quantum automorphism group. Indeed, since $QAut_{Bic}(K_2^c \circ K_2) \cong QAut_{Bic}(C_4)$ is commutative, while $QAut_{Bic}(K_2^c) \wr_{*} QAut_{Bic}(K_2) \cong C(H_2^+)$ is non-commutative, there does not exist any surjective *-homomorphism from $QAut_{Bic}(K_2^c \circ K_2)$ to $QAut_{Bic}(K_2^c) \wr_{*} QAut_{Bic}(K_2)$. 

\begin{rem} \label{Rem_Bic_noncomm_graph_prod}
The importance of the above proposition is the following:\\
If at least one of $QAut_{Bic}(\Gamma_1)$ or $QAut_{Bic}(\Gamma_2)$ is non-commutative, then Proposition \ref{Bic_noncomm_prod_graph} implies that $QAut_{Bic}(.)$ remains non-commutative for each of the product graphs $ \Gamma_1 \square \Gamma_2, \Gamma_1 \times \Gamma_2, \Gamma_1 \boxtimes \Gamma_2$ and $\Gamma_1 \circ \Gamma_2$.\\
 However, the converse does not hold: there exist graphs $\Gamma_i$ ($i=1,2$) for which each of $QAut_{Bic}(\Gamma_i)$ is commutative, while the quantum automorphism groups $QAut_{Bic}(.)$ corresponding to their product graphs $(\Gamma_1 \square \Gamma_2, \Gamma_1 \times \Gamma_2, \Gamma_1 \boxtimes \Gamma_2$ and $\Gamma_1 \circ \Gamma_2)$ are non-commutative. For example, take $\Gamma_1= K_2^c$ and $\Gamma_2= K_2$, for which both $QAut_{Bic}(K_2)$ and $QAut_{Bic}(K_2^c)$ are commutative, yet the Bichon's quantum automorphism groups $QAut_{Bic}(.)$ for their product graphs are non-commutative, since $K_2^c \square K_2 \cong K_2^c \boxtimes K_2 \cong K_2^c \circ K_2 \cong 2P_1$ and $K_2^c \times K_2 \cong 4K_1$.
\end{rem}

\begin{ex}
	If at least one of $\Gamma_1$ or $\Gamma_2$ contains non-trivial edge-free disjoint automorphisms (for example, a tree with at least 2 cherries), then Propositions \ref{Bic_noncomm} and \ref{Bic_noncomm_prod_graph} imply that all the product graphs have non-commutative quantum automorphism groups of Bichon.\\
\end{ex} 
 
\section{Commutativity on Quantum Atomorphism Group of Bichon} \label{Section_2}
In this section, we present several results related to the commutativity of quantum automorphism groups of graphs in the framework of Bichon.\\
It has been mentioned earlier that for any finite graph $\Gamma$, $QAut_{Ban}(\Gamma)  \approx QAut_{Ban}(\Gamma^c)$, whereas the same is not true in the sense of Bichon in general. 
On the other hand, if $\Gamma$ has no quantum symmetry, then it is evident that $QAut_{Bic}(\Gamma) \cong  QAut_{Bic}(\Gamma^c)$. But the converse is not true in general. Proposition \ref{Bic_comp_equal} provides some well-known families of graphs $\Gamma$ for which the converse is also true, i.e. 
\begin{equation}
`` QAut_{Bic}(\Gamma) \cong  QAut_{Bic}(\Gamma^c) \implies  \Gamma \text{ has no quantum symmetry} ".  \label{Stat_Bic=Bic^c_imply_non_qsymm}
\end{equation}

\noindent Before proving it, using Lemmas \ref{Ban=Bic} and \ref{Ban=Bic_implies_complement_comm}, we obtain the following:
\begin{lem}  \label{quad_free_Bic^c_comm}
	If $\Gamma$ is a graph which does not contain any quadrangle, then $QAut_{Bic}(\Gamma^c)$ is commutative.
\end{lem}
\noindent The above lemma yields examples of graphs whose Bichon's quantum automorphism groups are commutative: \\
$\bullet$ For any forest $F$ (in particular, for any tree), $QAut_{Bic}(F^c)$ is always commutative, whereas there exist many forests that admit quantum symmetries.\\
$\bullet$ Let $\widetilde{K_n}=(V(\widetilde{K_n}), E(\widetilde{K_n}))$ denotes a finite graph with $n (> 3)$ vertices such that $|E(\widetilde{K_n})| \geq \binom{n}{2}-3$. Then ${QAut}_{Bic}(\widetilde{K_n})$ is commutative. Nevertheless, many such graphs still admit quantum symmetries.

\begin{prop} \label{Bic_comp_equal}
	Let $\Gamma$ be a finite graph such that at least one of $\Gamma$ or $\Gamma^c$ contains no quadrangle.
    Then, $QAut_{Bic}(\Gamma) \cong  QAut_{Bic}(\Gamma^c)$ implies  $\Gamma$ has no quantum symmetry.
\end{prop}
\begin{proof}
Let $QAut_{Bic}(\Gamma) \cong  QAut_{Bic}(\Gamma^c)$. Since at least one of   $\Gamma$ and $\Gamma^c$ contains no quadrangle, Lemma \ref{Bic_comp_equal} implies that both $QAut_{Bic}(\Gamma)$ and $QAut_{Bic}(\Gamma^c)$ are commutative and thus isomorphic to $C(Aut(\Gamma))$. Again the hypothesis together with Lemma
\ref{Ban=Bic} ensures that $QAut_{Bic}(\Gamma) \cong  QAut_{Bic}(\Gamma^c) \cong QAut_{Ban}(\Gamma) \cong  QAut_{Ban}(\Gamma^c)$. Hence, $\Gamma$ has no quantum symmetry.
\end{proof}
\noindent Later, in Example \ref{Ex_graph_G'}, we construct an infinite family $\mathcal{G'}$ of graphs showing Proposition \ref{Bic_comp_equal} is not true without its stated hypothesis.
\begin{rem}
	It follows from the proof of Proposition \ref{Bic_comp_equal} that if $QAut_{Ban}(\Gamma) \cong QAut_{Bic}(\Gamma)$ (see also Lemma 3.6 and Lemma 6.1 of \cite{SS_Advances}), then Statement \eqref{Stat_Bic=Bic^c_imply_non_qsymm} holds as well. 
\end{rem}

\noindent Moreover,  it is evident from Theorem 3.15 of \cite{Lupini} that for almost all graphs $\Gamma$, $QAut_{Bic}(\Gamma) \cong QAut_{Bic}(\Gamma^c)$. In contrast, Proposition \ref{almost_all_tree} together with Lemma \ref{quad_free_Bic^c_comm} yields the following:
\begin{cor}
	For almost all trees $T$, $QAut_{Bic}(T) \ncong QAut_{Bic}(T^c)$.\\
\end{cor} 

From the above discussion, it follows that in order to identify a `non-trivial' family of graphs with commutative Bichon's quantum automorphism group, one should focus at least on those graphs $\Gamma$ satisfying the following properties:
\begin{enumerate}
	\item[(B1)] $\Gamma$ has quantum symmetry.
	\item[(B2)] $\Gamma^c$ contains a quadrangle.
	\item[(B3)] $QAut_{Bic}(\Gamma)$ is commutative.
\end{enumerate} 
We denote the class of graphs satisfying Properties (B1), (B2) and (B3) by $\mathcal{G}$. Note that if a graph $\Gamma$ fails to satisfy at least one of (B1) or (B2), then (B3) automatically holds. \\

\noindent We identify such graphs $\Gamma \in \mathcal{G}$ using the following proposition.
\begin{prop} \label{disjont_union_graphs}
	Let $\Gamma_1,..., \Gamma_n$ be the pairwise quantum non-isomorphic connected graphs, and let $\Gamma:= \sqcup_{i=1}^{n} k_i \Gamma_i$ (for $k_{i} \in \mathbb{N}$). Suppose that for each $i\in [n]$, $\Gamma_i$ does not contain a quadrangle or $\Gamma_i$ has no quantum symmetry. Then, $QAut_{Ban}(\Gamma) \approx QAut_{Bic}(\Gamma)$, and consequently, $QAut_{Bic}(\Gamma^c)$ is commutative.
\end{prop}
\begin{proof}
	Let $u$ and $v$ denote the canonical fundamental representation of $QAut_{Ban}(\Gamma)$ and $QAut_{Bic}(\Gamma)$ respectively, and let $|V(\Gamma_i)|=n_i$. By Lemma \ref{quantum_non_isomorphic 1}, $u$ reduces to the block diagonal matrices of the form $$\begin{bmatrix}
		u^{(1)} & 0 & \cdots & 0\\
		0 & u^{(2)} & \cdots & 0\\
		\vdots & \vdots & \ddots & \vdots\\
		0 & 0 & \cdots & u^{(n)}\\
	\end{bmatrix},$$ where each $u^{(i)}$ is a magic unitary matrix of order $k_i n_i$. Now, $A_{\Gamma}u=uA_{\Gamma}$ is equivalent to the relations: $$A_{k_i \Gamma_i} u^{(i)}=u^{(i)} A_{k_i \Gamma_i} \text{ for all } i \in [n].$$ 
If $\Gamma_j$ contains no quadrangle for some $j$, then the same holds for $k_j \Gamma_j$. Hence, by Lemma \ref{Ban=Bic}, we have $(QAut_{Bic}(k_j \Gamma_j), v^{(j)}) \approx (QAut_{Ban}(k_j \Gamma_j), u^{(j)})$, where $v^{(j)}$ denotes the canonical fundamental representation of $QAut_{Bic}(k_j \Gamma_j)$. If $\Gamma_j$ contains a quadrangle, then by hypothesis $QAut_{Ban}(\Gamma_j) \approx C(Aut(\Gamma_j)) \approx QAut_{Bic}(\Gamma_j)$, hence $x_{pq}$ commutes with $x_{rs}$ if $(p,r), (q,s) \in E(\Gamma_j)$, where $x$ denotes the canonical fundamental representation of $QAut_{Ban}(\Gamma_j)$ . We claim that $(QAut_{Bic}(k_j \Gamma_j), v^{(j)}) \approx (QAut_{Ban}(k_j \Gamma_j), u^{(j)})$. Since it is known from Theorem \ref{disjoint_union_k Gamma} that $(QAut_{Ban}(k_j \Gamma_j), u^{(j)}) \approx (QAut_{Ban}(\Gamma_j) \wr_{*} S_n^{+}, x \wr_{*} t) $, up to identification of fundamental representations $u^{(j)}$ with $x \wr_{*} t$, it is enough to verify that $x \wr_{*} t$ satisfies the Relation \eqref{Bic_comm_relation} of $QAut_{Bic}(k_j \Gamma_j)$.  Let $i \alpha$ denote the vertex $i$ in the $\alpha$-th component of $k_j \Gamma_j$. Therefore, if $(p \alpha, r \alpha)$ and $(q \beta, s \beta) \in E(k_j \Gamma_j)$, then $x_{pq}^{(\alpha)} t_{\alpha \beta} x_{rs}^{(\alpha)} t_{\alpha \beta} = x_{pq}^{(\alpha)} x_{rs}^{(\alpha)} t_{\alpha \beta} = x_{rs}^{(\alpha)} x_{pq}^{(\alpha)}t_{\alpha \beta}= x_{rs}^{(\alpha)} t_{\alpha \beta}x_{pq}^{(\alpha)} t_{\alpha \beta}$. Hence, $u^{(j)}$ satisfies all the defining relations of $QAut_{Bic}(k_{j} \Gamma_{j})$, which completes the proof of our claim.
\end{proof}

\noindent Therefore, all the graphs $\Gamma$ whose complement $\Gamma^c :=\sqcup_{i=1}^{n} k_i \Gamma_i$ satisfies the hypothesis of Proposition \ref{disjont_union_graphs} automatically satisfy Property (B3).\\
Furthermore, if we choose a graph $\Gamma$ whose complement graph $\Gamma^c$ is of the form $\sqcup_{i=1}^{n} k_i \Gamma_i$, where the graphs $\Gamma_{i}$'s are as described in Proposition \ref{disjont_union_graphs} that additionally fulfills the following conditions:
\begin{itemize}
	\item one of the connected components of $\Gamma^c$ contains a quadrangle, i.e. $\Gamma$ satisfies (B2); and
	\item at least two components of $\Gamma^c$ have a non-trivial automorphism group, or $k_i \geq 4$ (for some $i \in [n]$) (ensuring Property (B1));
\end{itemize}
then such a graph $\Gamma \in \mathcal{G}$. Hence, $|\mathcal{G}|$ is infinite.	\\

\noindent We now include several familiar connected graphs in $\mathcal{G}$, obtained via graph products. For this purpose, we first establish the following lemma, which states that if the surjective *-homomorphisms in Proposition 2.2.1 (i), (iii) of \cite{SSthesis} are actually isomorphisms (in Banica's framework), then the *-homomorphisms in Proposition \ref{Bic_noncomm_prod_graph} are also isomorphisms (in Bichon's framework).\\
 For simplicity, in the following lemma we use a single symbol $\circledast$ to represent the graph products defined in (i) and (iii) of Definition \ref{Def_Prod_graph}. Here, $u,v$ and $w$ (respectively, $\bar{u},\bar{v}$ and $\bar{w}$) denote the canonical fundamental representations of $QAut_{Ban}(\Gamma_1)$, $QAut_{Ban}(\Gamma_2)$ and $QAut_{Ban}(\Gamma_1 \circledast \Gamma_2)$ (respectively, $QAut_{Bic}(\Gamma_1)$, $QAut_{Bic}(\Gamma_2)$ and $QAut_{Bic}(\Gamma_1 \circledast \Gamma_2)$) respectively.
\begin{prop} \label{Prod_Ban_iso_imply_Bic_iso}
Let $\Gamma_1$ and $\Gamma_2$ be finite connected graphs, and let $\circledast$ denote any one of $\square$ or $\boxtimes$.
If the map $QAut_{Ban}(\Gamma_1 \circledast \Gamma_2) \to QAut_{Ban}(\Gamma_1) \otimes QAut_{Ban}(\Gamma_2)$ sending $w \to u \otimes v$ is an isomorphism, then the corresponding map $QAut_{Bic}(\Gamma_1 \circledast \Gamma_2) \to QAut_{Bic}(\Gamma_1) \otimes QAut_{Bic}(\Gamma_2)$ sending $\bar{w} \to \bar{u} \otimes \bar{v}$ is also an isomorphism.
\end{prop}
\begin{proof}
	By Proposition \ref{Bic_noncomm_prod_graph} (i) and (iii), we already have surjective *-homomorphisms from $\phi: QAut_{Bic}(\Gamma_1 \circledast \Gamma_2) \to QAut_{Bic}(\Gamma_1) \otimes QAut_{Bic}(\Gamma_2)$ such that $\bar{w} \mapsto \bar{u} \otimes \bar{v} $. For the other direction, since $(QAut_{Bic}(\Gamma_1 \circledast \Gamma_2), \bar{w})$ is a quantum subgroup of $(QAut_{Ban}(\Gamma_1 \circledast \Gamma_2), w)$, using the hypothesis, we obtain a surjective *-homomorphism from 
	$$ \theta: QAut_{Ban}(\Gamma_1) \otimes QAut_{Ban}(\Gamma_2) \to QAut_{Ban}(\Gamma_1 \circledast \Gamma_2) \to QAut_{Bic}(\Gamma_1 \circledast \Gamma_2)$$ such that $$u \otimes v ~~~~~~ \longmapsto ~~~~~~ w ~~~~~~ \longmapsto ~~~~~~~ \bar{w}.$$
	Now, consider the ideal $J$ generated by the elements of the form $(u_{ij}u_{kl}-u_{kl}u_{ij})$ and $(v_{\alpha \beta}v_{\gamma \delta} - v_{\gamma \delta}v_{\alpha \beta})$ whenever $(i,k), (j,l) \in E(\Gamma_1)$ and $(\alpha, \gamma), (\beta, \delta) \in E(\Gamma_2)$. We first show that
	
	\begin{enumerate}
		\item[(a)] $\theta(u_{ij}u_{kl}-u_{kl}u_{ij})=0$ if $(i,k) \in E(\Gamma_1)$ and $(j,l) \in E(\Gamma_1)$;
		\item[(b)] $\theta(v_{\alpha \beta}v_{\gamma \delta} - v_{\gamma \delta}v_{\alpha \beta})=0$ if $(\alpha, \gamma) \in E(\Gamma_2)$ and $(\beta, \delta) \in E(\Gamma_2)$.
	\end{enumerate}	 

	\noindent To show (a), let $(i,k) \in E(\Gamma_1)$ and $(j,l) \in E(\Gamma_1)$. Then, by the definition of Cartesian and strong product, $((i,\alpha), (k, \alpha)) \in E(\Gamma_1 \circledast \Gamma_2)$ and $((j,\beta), (l, \beta)) \in E(\Gamma_1 \circledast \Gamma_2)$ for all $\alpha, \beta \in V(\Gamma_2)$. Therefore, for any $\alpha, \beta \in V(\Gamma_2)$, we have
	\begin{align*}
	 & w_{(i, \alpha)(j, \beta)} w_{(k, \alpha)(l, \beta)} = w_{(k, \alpha)(l, \beta)}w_{(i, \alpha)(j, \beta)}\\
	\implies &  \theta(u_{ij} v_{\alpha \beta}u_{kl}v_{\alpha \beta}) = \theta(u_{kl} v_{\alpha \beta}u_{iju}v_{\alpha \beta}) \\
	\implies & \theta(u_{ij} u_{kl}v_{\alpha \beta}) = \theta(u_{kl}u_{ij}v_{\alpha \beta}).
	\end{align*}	
	Taking the sum over $\beta \in V(\Gamma_2)$ on both sides and using the linearity of $\theta$, we get $\theta(u_{ij} u_{kl} - u_{kl}u_{ij})=0$. The proof of (b) is similar. Thus, $J \subset ker\theta$.\\
	It is straightforward to observe that $(QAut_{Ban}(\Gamma_1) \otimes QAut_{Ban}(\Gamma_2)/J, \overline{u \otimes v}) \approx (QAut_{Bic}(\Gamma_1) \otimes QAut_{Bic}(\Gamma_2), \bar{u} \otimes \bar{v})$, where $\overline{u \otimes v}$ denotes the image of $u \otimes v$ under the canonical quotient map. Now, by Theorem 2.11 of \cite{Wangfree}, there exists a surjective *-homomorphism
	$$ \bar{\theta} : QAut_{Bic}(\Gamma_1) \otimes QAut_{Bic}(\Gamma_2) \to QAut_{Bic}(\Gamma_1 \circledast \Gamma_2) $$ sending $\bar{u} \otimes \bar{v} \mapsto \bar{w}$, which serves as the inverse of $\phi$.

\end{proof}

	

\begin{rem} \label{Rem_spetra}
	Theorem 2.2.2 (i) and (iii) of \cite{SSthesis} ensure that under certain hypotheses on spectra of the adjacency matrices of individual graphs, the desired isomorphism holds at the Banica level. Consequently, Proposition
	\ref{Prod_Ban_iso_imply_Bic_iso} implies that under the same hypothesis, we also have $$QAut_{Bic}(\Gamma_1 \circledast \Gamma_2) \cong QAut_{Bic}(\Gamma_1) \otimes QAut_{Bic}(\Gamma_2) ~~ \text{ for } ~~ \circledast = \square \text{ or } \boxtimes .$$
	We refer the reader to Theorem 4.1 of \cite{Banver} and Theorem 2.2.2 of \cite{SSthesis} for further details.
\end{rem} 
\begin{ex}
We provide two families of graphs constructed by Cartesian product and strong product.\\  
\vspace{0.01cm}
\noindent (1) Consider the family of graphs $K_m \square K_n$, where $m \neq n$ with at least one of $m$ or $n$ being at least 4 and the other being at least 2. Since the spectra $Sp(K_n)=\{-1, n-1\}$ for all $n \in \mathbb{N}$, the hypothesis of Theorem 2.2.2 (i) of \cite{SSthesis} is clearly satisfied. Thus, by Remark \ref{Rem_spetra}, $QAut_{Bic}(K_m \square K_n) \cong QAut_{Bic}(K_m) \otimes QAut_{Bic}(K_n)$ is commutative, as $QAut_{Bic}(K_n)$ is commutative for all $n \in \mathbb{N}$. At least one of $m$ or $n$ being $\geq 4$ guarantees that $K_m \square K_n$ has quantum symmetry (by Theorem 2.2.2 (i) of \cite{SSthesis}). Moreover, if one of $m,n \geq 4$ and the other $\geq 2$, then $(K_m \square K_n)^c$ contains a quadrangle. Therefore, the aforementioned family of graphs of the form $K_m \square K_n$ contained in $\mathcal{G}$. In particular, the prism graphs $Pr(K_n):=K_2 \square K_n \in \mathcal{G}$ for $n \geq 4$.\\
Similarly, by Corollary 4.1 (4) of \cite{Banver} and Remark \ref{Rem_spetra}, the prism graph $Pr(C_4):=K_2 \square C_4$ is also contained in $\mathcal{G}$. \\  \vspace{0.01cm}
\noindent (2) Consider the graphs $C_4 \boxtimes C_n$. Since $Sp(C_n)=\{2cos(\frac{2\pi k}{n}): k \in [n]\}$, there exist sufficiently many integers $n \geq 5$ that satisfy the eigenvalue condition in Theorem 2.2.2 (iii) of \cite{SSthesis}. Hence, by Remark \ref{Rem_spetra}, we obtain
$QAut_{Bic}(C_4 \boxtimes C_n) \cong QAut_{Bic}(C_4) \otimes QAut_{Bic}(C_n)$. The commutativity of $QAut_{Bic}(C_n)$ (for any $n \in \mathbb{N}$) and non-commutativity of $QAut_{Ban}(C_4)$ ensure that each $QAut_{Bic}(C_4 \boxtimes C_n)$ is commutative, whereas $QAut_{Ban}(C_4 \boxtimes C_n)$ is non-commutative. Thus, $C_4 \boxtimes C_n \in \mathcal{G}$ for sufficiently many integers $n \geq 5$.
\end{ex}

It is noteworthy that $\mathcal{G}$ is not closed under graph complementation. In fact, every graph $\Gamma \in \mathcal{G}$ constructed using Proposition \ref{disjont_union_graphs} has non-commutative $QAut_{Bic}(\Gamma^c)$. 
We now introduce a subclass $\mathcal{G}' \subset \mathcal{G}$ consisting of graphs $\Gamma$ that satisfy conditions (B1), (B2), and (B3) and additionally have the property that $QAut_{Bic}(\Gamma^c)$ is commutative. Note that $\mathcal{G}'$ is closed under graph complementation. Interestingly, every such $\Gamma \in \mathcal{G}'$ provides a counter-example to Proposition \ref{Bic_comp_equal}, in the sense that, without the hypothesis of the proposition, Statement \eqref{Stat_Bic=Bic^c_imply_non_qsymm} does not hold. The following examples demonstrate that $\mathcal{G'}$ is non-empty and, in fact, contains infinitely many graphs.
\begin{ex} \label{Ex_graph_G'}
	Consider a family of graphs $\Gamma_{C_4, P_{n-1}}$ (for $n \geq 2$), with vertex labeling as illustrated in Figure \ref{Gamma_C4_Pn}. Clearly, each graph $\Gamma_{C_4, P_{n-1}}$ contains a quadrangle. Furthermore, $QAut_{Ban}(\Gamma_{C_4, P_{n-1}})$ is non-commutative, since the graph admits two disjoint automorphisms $\sigma:=(a ~ b)$ and $\tau:=(1 ~ 1')(2 ~ 2') \cdots (n ~ n')$ in  $Aut(\Gamma_{C_4, P_{n-1}})$.\\
	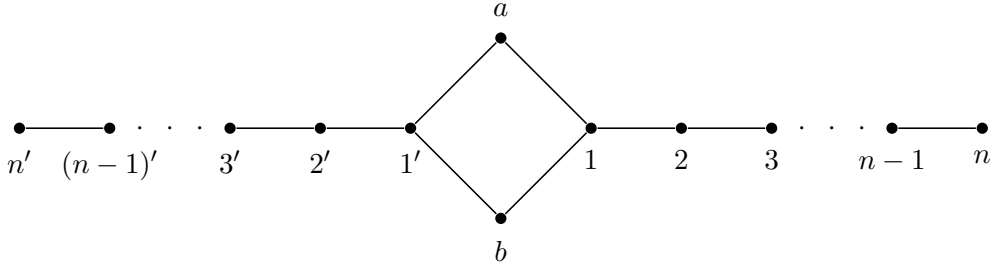
\begin{figure}[htpb]
	\centering
	\begin{tikzpicture}[
		vertex/.style={circle, fill, inner sep=1.5pt},
		every path/.style={semithick}
		]
		
		\foreach \x in {1,...,3} {
			\node[vertex] (l-\x) at (-\x*1.2, 0) {};
		}
		\node[vertex] (ln-1) at (-5.2, 0) {};
		\node[vertex] (ln) at (-6.4, 0) {};
		
		\draw (l-1) -- (l-2);
		\draw (l-2) -- (l-3);

		\path (-4.0,0) node {$\cdot$}
		(-4.4,0) node {$\cdot$}
		(-4.8,0) node {$\cdot$};
		
		\draw (ln-1) -- (ln);
		
		\node[below=2pt] at (l-1.south) {$1'$};
		\node[below=2pt] at (l-2.south) {$2'$};
		\node[below=2pt] at (l-3.south) {$3'$};
		\node[below=2pt] at (ln-1.south) {$(n-1)'$};
		\node[below=2pt] at (ln.south) {$n'$};
		
		\foreach \x in {1,...,3} {
			\node[vertex] (r-\x) at (\x*1.2, 0) {};
		}
		\node[vertex] (rn-1) at (5.2, 0) {};
		\node[vertex] (rn) at (6.4, 0) {};
		
		\draw (r-1) -- (r-2);
		\draw (r-2) -- (r-3);

		\path (4.0,0) node {$\cdot$}
		(4.4,0) node {$\cdot$}
		(4.8,0) node {$\cdot$};
		
		\draw (rn-1) -- (rn);
		
		\node[below=2pt] at (r-1.south) {$1$};
		\node[below=2pt] at (r-2.south) {$2$};
		\node[below=2pt] at (r-3.south) {$3$};
		\node[below=2pt] at (rn-1.south) {$n-1$};
		\node[below=2pt] at (rn.south) {$n$};
		
		\node[vertex] (a) at (0, 1.2) {};
		\node[above=2pt] at (a.north) {$a$};
		\node[vertex] (b) at (0, -1.2) {};
		\node[below=2pt] at (b.south) {$b$};
		
		\draw (l-1) -- (a);
		\draw (l-1) -- (b);
		
		\draw (r-1) -- (a);
		\draw (r-1) -- (b);		
	\end{tikzpicture} 
	\caption{$\Gamma_{C_4, P_{n-1}}$}     \label{Gamma_C4_Pn}
	\end{figure}
	 We establish that both  $QAut_{Bic}(\Gamma_{C_4, P_{n-1}})$ and $QAut_{Bic}(\Gamma_{C_4, P_{n-1}}^{c})$ are commutative. We denote the sets $\{1,2, ..., n\}$ and $\{1', 2', ..., n'\}$ by $[n]$ and $[n]'$, respectively. Let $i,j \in [n]$ and $j' \in [n]'$ with $j > i$. Let $u$ denote the canonical fundamental representation of  $QAut_{Bic}(\Gamma_{C_4, P_{n-1}})$. For all $q \in C_{r}(i)$ with $r:=n+i$, $deg(q)=1$. Moreover, there exist vertices $p_{j}:= (n-j+i)'$ and $p_{j'}:=(n-j+i)$ such that $d(j,p_{j})=n+i=d(j',p_{j'})$ with both $deg(p_j), deg(p_{j'}) > 1$. Then, by Lemma \ref{2}, we conclude that $u_{ij}=0=u_{ij'}$. A similar argument shows $u_{i'j}=0=u_{i'j'}$ for $i', j' \in [n]'$ and $j \in [n]$ with $j>i$. \\
	We also claim that $u_{aj}=u_{bj}=0=u_{aj'}=u_{bj'}$ for all $j \in [n]$ and $j' \in [n]'$. To this end, note that $deg(q)=1$ for all $q \in C_{n}(a) \cup C_{n}(b)$; however, for each $j \in [n]$ and $j' \in [n']$, there exist $r_{j}$ and $r_{j'}$ such that $d(j,r_{j})=n=d(j',r_{j'})$ with both $deg(r_j), deg(r_{j'}) > 1$. Thus, we conclude our claim using Lemma \ref{2} again.\\
	Therefore, since $u_{aa}+u_{ab}=1$ and $u_{jj}+u_{jj'}=1$ for all $j \in [n]$; with respect to the ordering $\{a, b, 1, 1', 2, 2', ..., n, n'\}$ of the vertices, $u$ reduces to the form 	\begin{equation} \begin{bmatrix}
		u_{aa} & 1-u_{aa} & 0 & 0 & \cdots & 0 & 0 \\
		1-u_{aa} & u_{aa} & 0 & 0 & \cdots & 0 & 0 \\
		0 & 0 & u_{11} & 1-u_{11} & \cdots &  0 & 0 \\
		0 & 0 & 1-u_{11} & u_{11} & \cdots & 0 & 0 \\
		\vdots & \vdots & \vdots & \vdots  & \ddots &  \vdots  & \vdots \\
		0 & 0 & 0 & 0 & \cdots & u_{nn} & 1-u_{nn} \\
		0 & 0 & 0 & 0 & \cdots & 1-u_{nn} & u_{nn} \\
	\end{bmatrix}. \label{Mat_Ex_C_4,P_n} \end{equation}
 Moreover, $(i, i+1) \in E(\Gamma_{C_4, P_{n-1}}) $ and $(i,(i+1)') \notin  E(\Gamma_{C_4, P_{n-1}})$ implies that
\begin{align*}
	& u_{ii}u_{(i+1)(i+1)}  =u_{ii} (1-u_{(i+1)(i+1)'}) = u_{ii} \\
 \text{ and \ \ \ \ }	& u_{ii}u_{(i+1)(i+1)}  = (1-u_{ii'}) u_{(i+1)(i+1)} = u_{(i+1)(i+1)}.
\end{align*} 
Therefore, we obtain: 
\begin{equation}
u_{11} = u_{22} = \cdots = u_{nn}. \label{Eqn_Ex_all_uii_equal}
\end{equation}
It is worth mentioning that, up to this point, all arguments are a consequence of Relations \eqref{proj}–\eqref{comm_with_adj}. \\
Now, since $(a,1) \in E(\Gamma_{C_4, P_{n-1}})$, using Relation \eqref{Bic_comm_relation}, we have  $u_{aa}u_{11}=u_{11}u_{aa}$. Hence, $QAut_{Bic}(\Gamma_{C_4, P_{n-1}})$ is commutative.\\ 
If $u^c$ denotes the canonical fundamental representation of $QAut_{Bic}(\Gamma_{C_4, P_{n-1}}^c)$, then the previous computations up to Equation \eqref{Eqn_Ex_all_uii_equal}, along with Remark \ref{Rem_Bic_qsubgrp_Ban^c}(b), ensure that $u^c$ has a matrix form similar to that in \eqref{Mat_Ex_C_4,P_n} with the additional relations that $u_{11}^c = u_{22}^c = \cdots = u_{nn}^c$. Now, since $(a,2) \in E(\Gamma_{C_4, P_{n-1}}^c)$, it follows from Relation \eqref{Bic_comm_relation} of $QAut_{Bic}(\Gamma_{C_4, P_{n-1}}^c)$ that $u_{aa}^{c} u_{22}^{c} = u_{22}^{c} u_{aa}^{c}$, and hence $QAut_{Bic}(\Gamma_{C_4, P_{n-1}}^c)$ is commutative.\\
\noindent Therefore, all the above facts together imply that both $\Gamma_{C_4, P_{n-1}}$ and $\Gamma_{C_4, P_{n-1}}^c$ belong to $\mathcal{G}'$.
\end{ex}

\begin{ex} \label{Ex_graph_G'_2}
	Let $\Gamma_{qa}$ be a connected quantum asymmetric graph (i.e. $QAut_{Ban}(\Gamma_{qa}) \cong \mathbb{C}$) containing no vertices of degree 2 (such graphs exist using Theorem 4.2.1 of \cite{Junk}), where the vertex set $V(\Gamma_{qa})$ is labeled as $\{1,2, ...,n\}=:[n]$ (for $n \geq 2$). Take two copies of $\Gamma_{qa}$ and label the corresponding vertices of the second copy of $\Gamma_{qa}$ as $\{1', 2', ...,n'\}=: [n]'$. Consider two new vertices $v$ and $w$. Define a new graph $\Gamma_{qa, C_4}$ with vertex set  $V:=\{v,w\} \sqcup [n] \sqcup [n]' $ and edge set $E$ which consists of all the edges in $E(2 \Gamma_{qa})$ together with four new edges $\{(v,1), (w,1), (v,1'), (w,1')\}$. By construction, $v$ and $w$ are the only two vertices of degree 2 in $\Gamma_{qa, C_4}$. Let $u$ denote the canonical fundamental representation of $QAut_{Bic}(\Gamma_{qa, C_4})$ with respect to the ordering of the vertices $\{v, w; 1, 2, ..., n; 1', 2', ..., n'\}$. Then, using Lemma \ref{6}, $u$ can be reduced to the form 
	\begin{equation} \label{Matrix_asymm ex_1}
\begin{bmatrix}
u_{vv} & 1-u_{vv} & 0_{1,2n} \\
1-u_{vv} & u_{vv} & 0_{1,2n} \\
0_{2n,1} & 0_{2n,1} & u' \\
\end{bmatrix},
	\end{equation}
	where $u'$ is a magic unitary matrix of order $2n \times 2n$ corresponding to the vertices in $[n] \sqcup [n]'$. Now, $u$ commutes with the adjacency matrix $$A_{\Gamma_{qa, C_4}}= \begin{bmatrix}
	0 & 0 & E_{11} & E_{11}\\
	0 & 0 & E_{11} & E_{11}\\
	E_{11}^t & E_{11}^t & A_{\Gamma_{qa}} & 0_{n} \\
	E_{11}^t & E_{11}^t & 0_{n} & A_{\Gamma_{qa}} \\
	\end{bmatrix}$$ (where $E_{11}$ denotes the $1 \times n$ matrix whose first entry is 1 and all other entries are 0) in particular implies $u'$ commutes with the adjacency matrix of $2 \Gamma_{qa}$. Since $\Gamma_{qa}$ is quantum asymmetric, using Theorem \ref{disjoint_union_k Gamma}, one can conclude that the magic unitary $u'$ must have the block form
	\begin{equation} \label{Matrix_asymm ex_2}
    \begin{bmatrix}
    u_{11} I_n & (1-u_{11}) I_n \\
    (1-u_{11}) I_n & u_{11} I_n \\
    \end{bmatrix}.
	\end{equation} 	
	It is worth noting that, by Remark \ref{Rem_Bic_qsubgrp_Ban^c}, a similar computation ensures that the canonical fundamental representation of $QAut_{Bic}(\Gamma_{qa, C_4}^c)$ must also take the same form as in \eqref{Matrix_asymm ex_1} and \eqref{Matrix_asymm ex_2}.\\ 
	For the graph $\Gamma_{qa, C_4}$, $(1,v) \in E$  implies $u_{vv}u_{11}=u_{11}u_{vv}$ (by Relation \eqref{Bic_comm_relation}), and hence  $QAut_{Bic}(\Gamma_{qa, C_4})$  is commutative.
	Similarly, the commutativity of $QAut_{Bic}(\Gamma_{qa, C_4}^c)$ follows, as $(2,v) \in E(\Gamma_{qa, C_4}^c)$. \\
	On the other hand, clearly $QAut_{Ban}(\Gamma_{qa, C_4})$ is non-commutative, since $\sigma:=(v,w)$ and $\tau:=(1 ~ 1') (2 ~ 2') \cdots (n ~ n')$ are non-trivial disjoint automorphisms. In fact, the above calculations show that the underlying  $C^*$-algebra of $QAut_{Ban}(\Gamma_{qa, C_4})$ is the unital universal $C^*$-algebra generated by two distinct independent projections. 
\end{ex}

We now present a few additional results that can sometimes help determine the commutativity of Bichon’s quantum automorphism group, but which do not hold in Banica’s framework.
\begin{lem}\label{Lem_ deg_n-1, n-2}
	Let $\Gamma$ be a finite graph on $n$ vertices with $i, j, k, l \in V(\Gamma)$, and let $u$ be the canonical fundamental representation of ${QAut}_{Bic}(\Gamma)$. If at least one of $i,j,k,l \in V(\Gamma)$ has degree $n-1$ or $n-2$ in $\Gamma$, then $u_{ij}$ commutes with $u_{kl}$.
\end{lem}
\begin{proof}
	To prove the commutativity of $u_{ij}$ and $u_{kl}$, it suffices by Relations \eqref{Bic_comm_relation} - \eqref{adj_nonadj_prod=0_eq2} to show that 
	$$
	u_{ij} u_{kl} = u_{kl} u_{ij} \text{ if both } (i,k) \text{ and } (j,l) \notin E(\Gamma).
	$$
	Let $(i,k) \text{ and } (j,l) \notin E(\Gamma)$. If $i=k$ or $j=l$, then clearly $u_{ij} u_{kl} = u_{kl} u_{ij}$, since $u$ is a magic unitary matrix. Hence, without loss of generality, we may further assume that $i \neq k$ and $j \neq l$. Therefore, in the case whenever at least one of $i,j,k,l \in V(\Gamma)$ has degree $n-1$, there is nothing further to prove. Let us focus on the case where at least one of $i,j,k,l \in V(\Gamma)$ has degree $n-2$.\\
	If $deg_{\Gamma}(l)=n-2$, then 
	\begin{align*}
u_{ij} u_{kl}= u_{ij} u_{kl} \left(\sum\limits_{r \in V(\Gamma)} u_{ir}\right) =& u_{ij} u_{kl}(u_{ij} + u_{il}) \text{ [by Proposition \ref{adj_nonadj_prod=0}]}\\
=& u_{ij}u_{kl}u_{ij}.
	\end{align*}
	Thus, the commutativity follows from Lemma \ref{6}.\\
	Next, if $deg_{\Gamma}(k)=n-2$, by the previous case, we have $u_{ji} u_{lk} = u_{lk}u_{ji}$. Applying the antipode $\kappa$ to both sides, we get $u_{ij} u_{kl} = u_{kl} u_{ij}$.\\
	A similar argument also applies when $deg(j)=n-2$ or $deg(i)=n-2$, and hence $u_{ij}u_{kl}=u_{kl}u_{ij}$ in all cases.
\end{proof}
The above Lemma \ref{Lem_ deg_n-1, n-2} is not true for vertices of degree $n-3$. For example, if we take the graph $\Gamma=2 K_2$ on 4 vertices, then all the vertices have degree 1, but $QAut_{Bic}(\Gamma)$ is non-commutative.
\begin{rem}
The Lemma \ref{Lem_ deg_n-1, n-2} does not hold for Bianca's version of the quantum automorphism group. For instance, $QAut_{Ban}(K_4)$ is non-commutative even though each vertex of $K_4$ has degree $3$, and $QAut_{Ban}(C_4)$ is non-commutative where each vertex of $C_4$ has degree $2$.
\end{rem}

The following proposition ensures that if a graph $\Gamma$ contains vertices of degree $n-1$ or $n-2$, then to determine the commutativity of $QAut_{Bic}(\Gamma)$, it suffices to check the commutativity of $QAut_{Bic}(.)$ for the induced subgraph obtained by removing all vertices of degree $n-1$ and $n-2$.
\begin{prop}{\label{induced_Bic_comm}}
	Let $\Gamma=(V(\Gamma), E(\Gamma))$ be a graph with $n$ vertices.
	Define $A_{n-1}:=\{i\in V(\Gamma)|\, deg(i)= n-1\}$, $A_{n-2}:=\{i\in V(\Gamma)|\, deg(i)= n-2\}$. For $S \in \{A_{n-1}, A_{n-2}, A_{n-1} \cup A_{n-2}\}$, let $\Gamma_{S}$ denote the induced subgraph obtained by removing the set $S \subset V(\Gamma)$.\\
	If $QAut_{Bic}(\Gamma_S)$ is commutative, then so is $QAut_{Bic}(\Gamma)$.
\end{prop}
\begin{proof}
	Let $u$ and ${u^S}$ denote the canonical fundamental representations of ${QAut}_{Bic}(\Gamma)$ and ${QAut}_{Bic}(\Gamma_S)$ respectively.\\
	Assume that ${QAut}_{Bic}(\Gamma_S)$ is commutative.
	Using the Proposition \ref{3}, we can conclude that $u$ is a block diagonal matrix consisting of 3 blocks corresponding to the elements of $A_{n-1}$, $A_{n-2}$ and $V(\Gamma)-[A_{n-1} \cup A_{n-2}]$ respectively, as follows:
	
	$$
	u := 
	\begin{bmatrix}
	u^{(1)} & 0 & 0 \\
	0 & u^{(2)} & 0 \\
	0 & 0 & u^{(12)}
	\end{bmatrix},
	$$	
	where $u^{(1)}$, $u^{(2)}$ and $u^{(12)}$ denote the matrix corresponding to the elements of $A_{n-1}$,	$A_{n-2}$ and $V(\Gamma)-[A_{n-1} \cup A_{n-2}]$ respectively.
	Consider a $*$-subalgebra $\mathcal{A}_{S}$ of $QAut_{Bic}(\Gamma)$ generated by $\{u_{ij} : i,j \in V(\Gamma) - S \}$. 
	Since $\Gamma_S$ is the induced subgraph of $\Gamma$, then it is easy to verify $u_{ij}$'s for $i,j \in V(\Gamma) - S$ satisfy the defining relations \eqref{proj} - \eqref{row_sum_1} and \eqref{Bic_comm_relation} - \eqref{adj_nonadj_prod=0_eq2}  of  
	$QAut_{Bic}(\Gamma_{S})$. Hence, by universal property there exists a surjective $*$-homomorphism from $QAut_{Bic}(\Gamma_S)$ to $
	\mathcal{A}_S$ such that $u^{S}_{ij} \mapsto u_{ij}$ for all $i, j \in V(\Gamma)-S$. Hence, the commutativity of $QAut_{Bic}(\Gamma_S)$ ensures that $u_{ij} u_{kl} = u_{kl} u_{ij}$ for any $i, j, k, l \in V(\Gamma)-S$. In particular, all the entries of $u^{(12)}$ commute with each other.\\
	Now, from Lemma \ref{Lem_ deg_n-1, n-2}, it automatically follows that each of the entries of the matrices $u^{(1)}$ and $u^{(2)}$ commutes with all the elements of $u$. Therefore, ${QAut}_{Bic}(\Gamma)$ is commutative.\\
\end{proof}

	The converse of Proposition \ref{induced_Bic_comm}  is not true. For example, if we take the graph $\Gamma$ on 6 vertices shown in Figure \ref{small gamma non-c}, then removing all the vertices of degree 5 (specifically, $\{5,6\}$), we obtain the induced subgraph $\Gamma_S \cong P_1 \sqcup P_1$ having non-commutative $QAut_{Bic}(\Gamma_S)$.\\	
	\begin{figure}
		\begin{center}
		\begin{tikzpicture}[scale=1, every node/.style={scale=1}]
		\tikzstyle{vertex}=[circle, fill, inner sep=1.5pt]
		
		\node[vertex] (1) at (0,0) {};
		\node[below=2pt of 1] {$1$};
		
		\node[vertex] (2) at (1.5,0) {};
		\node[below=2pt of 2] {$2$};
		
		\node[vertex] (4) at (2.7,0) {};
		\node[below=2pt of 4] {$4$};
		
		\node[vertex] (3) at (4.5,0) {};
		\node[below=2pt of 3] {$3$};
		
		\node[vertex] (5) at (-1,-2.2) {};
		\node[below=2pt of 5] {$5$};
		
		\node[vertex] (6) at (-1,2.2) {};
		\node[above=2pt of 6] {$6$};
		
		\draw (1)--(2);
		\draw (4)--(3);
		\draw (5)--(1);
		\draw (5)--(2);
		\draw (5)--(3);
		\draw (6)--(1);
		\draw (6)--(2);
		\draw (6)--(3);
		\draw (6)--(4);
		\draw (6)--(5);
		
		\node at (1.5,-2.8) {$\Gamma$};
		
		\draw[-{Stealth[length=3mm]}] (5.5,0) -- (7,0);
		
		\node[vertex] (a1) at (8,0.8) {};
		\node[above=2pt of a1] {$1$};
		
		\node[vertex] (a2) at (9.5,0.8) {};
		\node[above=2pt of a2] {$2$};
		
		\node[vertex] (a3) at (8,-1) {};
		\node[below=2pt of a3] {$3$};
		
		\node[vertex] (a4) at (9.5,-1) {};
		\node[below=2pt of a4] {$4$};
		
		\draw (a1)--(a2);
		\draw (a3)--(a4);
		
		\node at (8.8,-2.8) {$\Gamma_{S}$};
		
		\end{tikzpicture}
		\end{center}
		\caption{} \label{small gamma non-c}
	\end{figure}
\noindent Since ${deg}(1) = {deg}(2) = {deg}(3) = 3, {deg}(4) = 2, {deg}(5) = 4 \text{ and } {deg}(6) = 5$, Proposition \ref{3} implies that the fundamental representation of $ QAut_{Bic}(\Gamma)$ is a block-diagonal matrix with three $1 \times 1$ blocks and one $3 \times 3$ block. The elements within the $3 \times 3$ magic unitary block commute with each other. Therefore, $QAut_{Bic}(\Gamma)$ is commutative.

\begin{rem}
	Proposition \ref{induced_Bic_comm} does not hold in Banica’s framework. For instance, consider the graph $K_4 \sqcup K_1$ on 5 vertices, for which $QAut_{Ban}(K_4 \sqcup K_1)$ is non-commutative. After removing all vertices of degree 3, the resulting induced subgraph is $K_1$, which has no quantum symmetry.\\
\end{rem}

\section{Some CMQG constructions arising as quantum automorphism groups of Bichon}
In this section, we establish that given a finite family of CMQGs $\{Q_i\}_{i=1}^{m}$ arising as quantum symmetries of certain graphs (in the sense of Bichon), there exist connected graphs $\Gamma_{free}$ and $\Gamma_{ten}$ having quantum symmetry groups $*_{i=1}^{m} Q_{i}$ and $\otimes_{i=1}^{m} Q_i$ respectively. Furthermore, it is shown that there exists a connected graph $\Gamma_{wr}$ whose quantum symmetry group is $Q_1 \wr_{*} Q_2$. We introduce two lemmas that will be used to construct such graphs. \\

\noindent In the following lemma, we demonstrate that for any disconnected graph, one can always construct a connected graph preserving its quantum automorphism group.
\begin{lem} \label{one_vertex_add}
	Given a finite disconnected graph $\Gamma$, there exists a finite connected graph $\Gamma'$ such that $QAut_{Bic}(\Gamma') \cong QAut_{Bic}(\Gamma)$.
\end{lem}
\begin{proof}
	Define a connected graph $\Gamma'= (V(\Gamma'), E(\Gamma'))$, by introducing a new vertex $v' \notin V(\Gamma)$, whose vertex set $V(\Gamma')=V(\Gamma)\sqcup \{v'\}$ and edge set $E(\Gamma)=E(\Gamma) \sqcup E_{v'}$, where $E_{v'}:= \{(v',v): v \in V(\Gamma)\}$. Let $u$ and $u'$ be the canonical fundamental representations of $QAut_{Bic}(\Gamma)$ and $QAut_{Bic}(\Gamma')$ respectively. Note that since $\Gamma$ is disconnected, we have $deg_{\Gamma'}(v')=|V(\Gamma)| > deg_{\Gamma'}(v)$ for all $v \in V(\Gamma)$. It follows from Proposition \ref{3} that  $u'_{v'v'}=1$, and hence $u'$ reduces to the block form $\begin{bmatrix}
	1 & 0_{1,n}\\
	0_{n,1} & u''
	 \end{bmatrix}$, where $u''=(u''_{vw})_{v,w \in V(\Gamma)}$ is a magic unitary matrix. Note that the adjacency matrix $A_{\Gamma'}$ is of the form $ \begin{bmatrix}
		0 & J_{1,n} \\
		J_{n,1} & A_{\Gamma}
	\end{bmatrix} $. Therefore, the relation $u' A_{\Gamma'}=A_{\Gamma'}u'$ is equivalent to $u'' A_{\Gamma}=A_{\Gamma}u''$. Furthermore, it is evident that $u_{vw}''$'s satisfy Relation \eqref{Bic_comm_relation} for each $v,w \in V(\Gamma)$. Thus, by the universal property of $QAut_{Bic}(\Gamma)$, there exists a surjective $C^*$-homomorphism from $QAut_{Bic}(\Gamma) \to QAut_{Bic}(\Gamma')$ sending $u_{vw} \mapsto u''_{vw}=u'_{vw}$ for all $v, w \in V(\Gamma)$.\\
	Conversely, extending $u$ to the magic unitary matrix $\begin{bmatrix}
	1 & 0_{1,n}\\
	0_{n,1} & u
	\end{bmatrix} \in M_{n+1}(QAut_{Bic}(\Gamma))$, it is straightforward to observe that the entries of the matrix satisfy all the defining relations \eqref{proj} - \eqref{Bic_comm_relation} for the graph $\Gamma'$. Hence, by the universal property of $QAut_{Bic}(\Gamma')$, there exists a surjective $C^*$-homomorphism $QAut_{Bic}(\Gamma') \to QAut_{Bic}(\Gamma)$ which maps $u'_{vw} \mapsto u_{vw}$ for all  $v,w \in V(\Gamma)$.
\end{proof}
\begin{rem} \label{Rem_Ban_one_vertex}
	The above argument also shows that for the same graph $\Gamma'$ as described in Lemma \ref{one_vertex_add}, we have $QAut_{Ban}(\Gamma') \cong QAut_{Ban}(\Gamma)$.\\
	However, in Banica's framework, due to Proposition \ref{Ban=Ban^c}, one may simply take $\Gamma'=\Gamma^c$, as $\Gamma^c$ is connected.\\
\end{rem}

\noindent The following lemma establishes that, for a given graph, there exists a graph (with connected complement) whose quantum automorphism groups remain invariant.
\begin{lem} \label{n_vertex_add}
	Given a graph $\Gamma$ on $n$ $(\geq 2)$ vertices, there exists a connected graph $\widetilde{\Gamma}$ on $2n$ vertices such that $QAut_{Bic}(\widetilde{\Gamma}) \cong QAut_{Bic}(\Gamma)$ and its complement graph ${\widetilde{\Gamma}}^c$ is connected.
\end{lem}
\begin{proof}
		Let $\Gamma$ be a connected graph with $V(\Gamma):=\{v_1,v_2,...,v_n\}$. Define a graph $\widetilde{\Gamma}= (V(\widetilde{\Gamma}), E(\widetilde{\Gamma}))$ by adjoining $n$ new  vertices $v_{n+1}, v_{n+2},..., v_{2n} \notin V(\Gamma)$, so that the vertex set $V(\widetilde{\Gamma})=V(\Gamma)\sqcup \{v_{i+n}: i \in [n] \}$ and the edge set $E(\widetilde{\Gamma})=E(\Gamma) \sqcup  \{(v_i,v_{i+n}): i \in [n] \}$. Observe that its complement graph ${\widetilde{\Gamma}}^c$ is always connected, and also note that $\widetilde{\Gamma}$ coincides with the corona product $\Gamma \odot K_1$ (see Definition \ref{Def_corona_prod}).
		
		\begin{figure}[htpb]
			\centering
	\begin{tikzpicture}[
	dot/.style={circle, fill, inner sep=1.5pt},
	node distance=1.5cm and 2cm
	]
	
	\node[dot, label=above:$v_{1}$] (t1) at (0, 2) {};
	\node[dot, label=below:$v_{n+1}$] (b1) at (0, 0) {};
	
	\node[dot, label=above:$v_{2}$, right=of t1] (t2) {};
	\node[dot, label=below:$v_{n+2}$, right=of b1] (b2) {};
	
	\node[dot, label=above:$v_{3}$, right=of t2] (t3) {};
	\node[dot, label=below:$v_{n+3}$, right=of b2] (b3) {};
	
	\node[right=of t3] (tdots) {$\dots$};
	\node[right=of b3] (bdots) {$\dots$};
	
	\node[dot, label=above:$v_{n}$, right=of tdots] (tn) {};
	\node[dot, label=below:$v_{2n}$, right=of bdots] (bn) {};
	
	\foreach \i in {1, 2, 3, n} {
		\draw (t\i) -- (b\i);
	}
	
	\foreach \i in {1, 2, 3, n} {
		\draw (t\i) -- ++(45:0.6cm);
	}

	\node at ($(tdots)+(-2,1)$) {Connected graph $\Gamma$};
	
	\node[
	draw,
	densely dashed,
	rounded corners=12pt,
	fit=(t1) (tn) (tdots),
	inner ysep=1.2cm, 
	inner xsep=0.5cm  
	] {};
	
	\end{tikzpicture}
	
			\caption{$\widetilde{\Gamma} = \Gamma \odot K_1$}  \label{Gamma_curl}
		\end{figure}
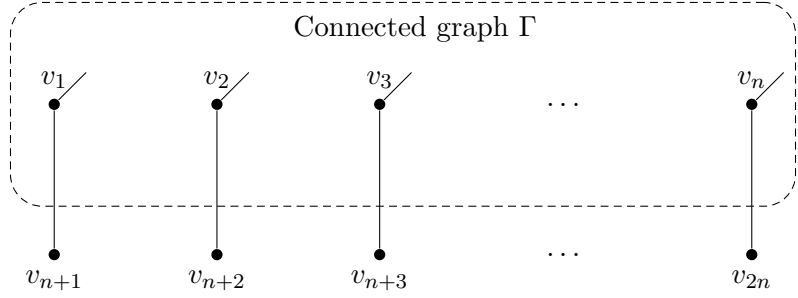
			
		 We claim that $QAut_{Bic}(\widetilde{\Gamma}) \cong QAut_{Bic}(\Gamma)$. To that end, let $\widetilde{u}$ and $u$ be the fundamental representations of $QAut_{Bic}(\widetilde{\Gamma})$ and $QAut_{Bic}(\Gamma)$ respectively. Since $deg_{\widetilde{\Gamma}}(v_i) > 1$ and $deg_{\widetilde{\Gamma}}(v_{n+i}) =1$ for all $i \in [n]$, the representing matrix $\widetilde{u}$ decomposes as $\begin{bmatrix}
		\widetilde{u}' & 0_{n}\\
		0_{n} & \widetilde{u}''
		\end{bmatrix}$, where $\widetilde{u}'$ and $\widetilde{u}''$ are two block matrices of order $n$. Note that $A_{\widetilde{\Gamma}}= \begin{bmatrix}
		A_{\Gamma} & I_n.  \\
		I_n & 0_n
		\end{bmatrix}.$
		Thus, the relation $\widetilde{u} A_{\widetilde{\Gamma}}=A_{\widetilde{\Gamma}} \widetilde{u}$ is equivalent to the relations $\widetilde{u}' A_{\Gamma}=A_{\Gamma}\widetilde{u}'$ and $\widetilde{u}' = \widetilde{u}''$. Hence, the fundamental representation $\widetilde{u}$ of $\widetilde{\Gamma}$ reduces to the form $\begin{bmatrix}
		\widetilde{u}' & 0_{n}\\
		0_{n} & \widetilde{u}'
		\end{bmatrix}$. Moreover, it is clear that the generators $({\widetilde{u}'}_{ij})_{i,j \in [n]}$ satisfy Relation \eqref{Bic_comm_relation} for the graph $\Gamma$.\\
		Conversely, consider the block magic unitary matrix $\begin{bmatrix}
			u & 0_n \\
			0_n & u
		\end{bmatrix} \in M_{2n}(QAut_{Bic}(\Gamma))$. Its entries satisfy all the defining relations \eqref{proj} - \eqref{Bic_comm_relation} for the graph $\widetilde{\Gamma}$.\\
		 Therefore, using the universal properties, we get the desired isomorphism $QAut_{Bic}(\widetilde{\Gamma}) \cong QAut_{Bic}(\Gamma)$.\\
		 
		If $\Gamma$ is disconnected, then we first apply Lemma \ref{one_vertex_add} to the graph $\Gamma$ to obtain a connected graph $\Gamma'$ having isomorphic quantum automorphism groups. Then, using the above corona product construction for the graph $\Gamma'$, we obtain the connected graph $\widetilde{\Gamma'}$ with the required properties.
\end{proof}

\noindent A repeated application of the above procedure yields the following result:

\begin{cor} \label{Infinitely _many}
		Given a finite graph $\Gamma$, there exist infinitely many finite connected graphs $\Gamma_{i}$ $(i \in \mathbb{N})$ with pairwise distinct orders such that $QAut_{Bic}(\Gamma_{i}) \cong QAut_{Bic}(\Gamma)$ for all $i \in \mathbb{N}$, and each complement graph $\Gamma_{i}^c$ $(i \in \mathbb{N})$ is connected.
\end{cor} 

\noindent It is worth mentioning that the graphs $\widetilde{\Gamma_i}$, including $\Gamma$ itself, are pairwise quantum non-isomorphic. Moreover, if the original graph $\Gamma$ is connected, then both $\widetilde{\Gamma}$ and ${\widetilde{\Gamma}}^c$ are connected. 
\begin{rem} \label{Rem_Ban_n vertex_add}
	The above construction of $\widetilde{\Gamma}$ also ensures that $QAut_{Ban}(\widetilde{\Gamma}) \cong QAut_{Ban}(\Gamma)$. Consequently, the statement of the Corollary \ref{Infinitely _many} remains valid for the quantum automorphism groups of graphs in the sense of Banica as well.\\
\end{rem}

\subsubsection{\bf Free product construction arising as quantum qutomorphism group of Bichon} \label{Section_Free product constuction}
It was shown in \cite{Qauttrees} that, for a given finite family of CMQGs $\{Q_i\}_{i=1}^{m}$ arising as the quantum automorphism groups of certain trees ${T_{i}}$, there exists a tree $T$ such that $QAut_{Ban}(T) \cong *_{i=1}^{m} ~~ QAut_{Ban}(T_i)$, and consequently, $QAut_{Bic}(T) \cong *_{i=1}^{m} ~~ QAut_{Bic}(T_i)$. Here, we demonstrate that the analogous results also hold (in the sense of both Banica and Bichon) if the family of trees is replaced by the family of finite graphs.

\begin{thm} \label{free_product_construction}
	Given a finite family of finite graphs $\{\Gamma_i\}_{i=1}^{m}$  $(m \geq 2)$, there exists a finite connected graph $\Gamma_{free}$ such that $QAut_{Bic}(\Gamma_{free}) \cong *_{i=1}^{m} ~~ QAut_{Bic}(\Gamma_i)$.
\end{thm}
\begin{proof}
	Without loss of generality, we can assume that all $\Gamma_{i}$'s are connected (by Lemma \ref{one_vertex_add}). If each $\Gamma_i$ is quantum non-isomorphic to $\Gamma_j$ for all $i \neq j$, then from Lemma \ref{quantum_non_isomorphic 2} and Remark \ref{quantum_non_isomrphic_Bic}, it follows that $QAut_{Bic}(\sqcup_{i=1}^{m} \Gamma_i) \cong *_{i=1}^{m} ~~ QAut_{Bic}(\Gamma_{i})$. Otherwise, repeatedly applying Lemma \ref{n_vertex_add} to the graphs $\Gamma_i$ $(i \in [m])$, if necessary, one can always construct $m$ distinct connected graphs $\Gamma^{(i)}$ (for $i \in [m]$) having distinct orders such that $QAut_{Bic}(\Gamma^{(i)}) \cong QAut_{Bic}(\Gamma_i)$ for all $i \in [m]$. It is possible because Corollary  \ref{Infinitely _many} ensures that for any given graph, one can obtain another graph of sufficiently large order with an isomorphic quantum automorphism group. Since $\Gamma^{(i)}$'s are quantum non-isomorphic to each other, using the aforementioned fact together with Remark \ref{quantum_non_isomrphic_Bic} implies that $QAut_{Bic}(\sqcup_{i=1}^{m} \Gamma^{(i)}) \cong *_{i=1}^{m} ~~ QAut_{Bic}(\Gamma_{i})$. Finally, using Lemma \ref{one_vertex_add}, one can produce a connected graph $\Gamma_{free}$ [$= (\sqcup_{i=1}^{m} \Gamma_{i})' \text{ or } (\sqcup_{i=1}^{m} \Gamma^{(i)})'$] such that $QAut_{Bic}(\Gamma_{free}) \cong *_{i=1}^{m} ~~ QAut_{Bic}(\Gamma_{i})$.
\end{proof}

\begin{rem}
	It is already mentioned in Remarks \ref{Rem_Ban_one_vertex} and \ref{Rem_Ban_n vertex_add} that analogous constructions remain valid to make Banica's quantum automorphism group invariant. Therefore, the proof of Theorem \ref{free_product_construction} directly yields the following:\\
	Given a finite family of finite graphs $\{\Gamma_i\}_{i=1}^{m}$  $(m \geq 2)$, there exists a finite connected graph $\Gamma_{free}$ such that $QAut_{Ban}(\Gamma_{free}) \cong *_{i=1}^{m} ~~ QAut_{Ban}(\Gamma_i)$. \\
\end{rem}

\subsubsection{\bf Tensor product construction arising as quantum automorphism group of Bichon} \label{Section_Tensor product constuction}
We construct a connected graph $\Gamma_{ten}$ whose quantum automorphism group (in the sense of Bichon) is isomorphic to the maximal tensor product of CMQGs that arise as Bichon's quantum automorphism groups of some graphs.

\begin{thm} \label{tensor_product_construction}
	Given a finite family of finite graphs $\{\Gamma_i\}_{i=1}^{m}$  $(m \geq 2)$, there exists a finite connected graph $\Gamma_{ten}$ such that $QAut_{Bic}(\Gamma_{ten}) \cong {\otimes}_{i=1}^{m} ~~ QAut_{Bic}(\Gamma_i)$.
\end{thm}
\begin{proof}
	The first step of the construction is quite similar to the free product case described in the proof of Theorem $\ref{free_product_construction}$. As before, by repeatedly applying Lemma \ref{n_vertex_add} (at least once) to each graph $\Gamma_i$ $(i \in [m])$, we can always construct $m$ pairwise quantum non-isomorphic connected graphs $\Gamma^{(i)}$ (for $i \in [m]$) having distinct orders such that each ${\Gamma^{(i)}}^c$ is connected and satisfying $QAut_{Bic}(\Gamma^{(i)}) \cong QAut_{Bic}(\Gamma_i)$ for all $i \in [m]$. Moreover, all ${\Gamma^{(i)}}^c$'s are pairwise quantum non-isomorphic.\\
	Next, we construct a graph $\Gamma_{ten}=(V(\Gamma_{ten}), E(\Gamma_{ten}))$ by joining the edges between all the vertices belonging to different connected components $\Gamma^{(i)}$, i.e. $$V(\Gamma_{ten}):= V(\sqcup_{i=1}^{m} \Gamma^{(i)})$$ and $$E(\Gamma_{ten}):= \left[ \sqcup_{i=1}^{m} E(\Gamma^{(i)}) \right] \bigsqcup \left\lbrace (v,w): v \in V(\Gamma^{(i)}) \text{ and } w \in V(\Gamma^{(j)}) \text{ for all } i \neq j \right\rbrace .$$
	Note that its complement $\Gamma_{ten}^c$ is isomorphic to  $ \sqcup_{i=1}^{m} {\Gamma^{(i)}}^c $, where ${\Gamma^{(i)}}^c$'s are pairwise quantum non-isomorphic. Let $u$ be the canonical fundamental representation of $QAut_{Bic}(\Gamma_{ten})$. Since $(QAut_{Bic}(\Gamma_{ten}), u)$ is the quantum subgroup of $(QAut_{Ban}(\Gamma_{ten}^c), u)$ and $QAut_{Ban}(\Gamma_{ten}^c) \cong *_{i=1}^{m} ~ QAut_{Ban}({\Gamma^{(i)}}^c)$ (by Lemma \ref{quantum_non_isomorphic 2}), it follows that $u$ decomposes as a block diagonal matrix $$\begin{bmatrix}
	u^{(1)} & 0 & 0 & \cdots & 0 \\
	0 & u^{(2)} & 0 & \cdots & 0 \\
	0 & 0 & u^{(3)} & \cdots & 0 \\
	\vdots & \vdots & \vdots & \ddots & 0\\
	0 & 0 & 0 & \cdots & u^{(m)} 
	\end{bmatrix},$$ where $u^{(i)}$ is a magic unitary matrix of order $n_i := |V(\Gamma^{(i)})|$, with all its entries belonging to $QAut_{Bic}(\Gamma_{ten})$ for every $i \in [m]$. Moreover, adjacency matrix $A_{\Gamma_{ten}}$ also forms a block matrix:
	 $$\begin{bmatrix}
	A_{\Gamma_1} & J_{n_1,n_2} & J_{n_1,n_3} & \cdots & J_{n_1,n_m} \\
	J_{n_2,n_1} & A_{\Gamma_2} & J_{n_2,n_3} & \cdots & J_{n_2,n_m} \\
	J_{n_3,n_1} & J_{n_3,n_2} & A_{\Gamma_3} & \cdots & J_{n_3,n_m} \\
	\vdots & \vdots & \vdots & \ddots & 0\\
	J_{n_m,n_1} & J_{n_m,n_2} & J_{n_m,n_3} & \cdots & A_{\Gamma_m}
	\end{bmatrix}.$$ Thus, the relation $A_{\Gamma_{ten}}u=uA_{\Gamma_{ten}}$ is equivalent to the relations: $$A_{\Gamma^{(i)}}u^{(i)}=u^{(i)}A_{\Gamma^{(i)}} ~~(i \in [m]).$$ 
     For a given $i \in [m]$, $(v,w) \in E(\Gamma^{(i)})$ implies $(v,w) \in E(\Gamma_{ten})$, and therefore the commutativity relation \eqref{Bic_comm_relation} holds among the entries $u^{(i)}_{vw}$ for $v,w \in V(\Gamma^{(i)})$. 
	Moreover, since every vertex $v \in V(\Gamma^{(i)})$ is adjacent to every vertex $w \in V(\Gamma^{(j)})$ for all $i \neq j$ in $\Gamma_{ten}$, the commutativity of $u^{(i)}_{vw}$ and $u^{(j)}_{v'w'}$ (for $v,w \in V(\Gamma^{(i)})$ and $v',w' \in V(\Gamma^{(j)})$ with $i \neq j$) directly follows from Relation \eqref{Bic_comm_relation}.\\
	 Now, using the definition of the tensor product together with the universal properties of the respective algebras, we obtain the desired surjective $C^*$-homomorphisms in both directions. Since $\Gamma_{ten}$ is connected, the proof is complete.	
\end{proof}

To construct a graph $\Gamma_{ten}$ as described in Theorem \ref{tensor_product_construction}, the first step is essential, that is, to construct graphs $\Gamma^{(i)}$ (for $i \in [m]$) with distinct orders, ensuring that all ${\Gamma^{(i)}}^c$'s are connected and mutually quantum non-isomorphic. Even if the given $\Gamma_i$'s are connected and pairwise quantum non-isomorphic, then directly joining all the edges between $\Gamma_i$ and $\Gamma_j$ for all $i \neq j$ does not necessarily fulfill our requirements. For instance, consider the graphs $\Gamma_1=K_1$ and $\Gamma_2=K_2$. Then, directly applying only the second step, we get a 3-cycle $K_3$, for which $QAut_{Bic}(K_3) \cong C(S_3)$, whereas $QAut_{Bic}(K_1) \otimes QAut_{Bic}(K_2) \cong \mathbb{C} \otimes C(S_2)$.  \\  

\subsubsection{\bf Free wreath product construction arising as quantum automorphism group of Bichon} \label{Section_free wreath product constuction}
The main focus in this section is that given two CMQGs $Q_1$ and $Q_2$, arising as quantum automorphism groups of finite graphs $\Gamma_1$ and $\Gamma_2$ respectively, there exists a finite connected graph $\Gamma_{wr}$ having quantum symmetry $Q_1 \wr_{*} Q_2$.\\
Now, in Theorem 1.1\footnote{The result stated in Theorem 1.1 of \cite{Lexico} is formulated using an alternative convention for the lexicographic product.} of \cite{Lexico} (see also Theorem 4.5 of \cite{lexico} for regular graphs), the quantum analogue of Sabidussi’s Theorem (in Banica’s setting) establishes that $$QAut_{Ban}(\Gamma_1 \circ \Gamma_2) \cong QAut_{Ban}(\Gamma_1) \wr_{*} QAut_{Ban}(\Gamma_2)$$ if and only if the following `Sabidussi' conditions are satisfied:
\begin{enumerate}
	\item[(i)] If $\Gamma_1$ is not connected, then $\Gamma_2$ has no twins, i.e. $N_{\Gamma_2}(v) \neq N_{\Gamma_2}(w)$ for all $v, w \in V(\Gamma_2)$.
	\item[(ii)] If $\Gamma_1^c$ is not connected, then $\Gamma_2^c$ has no twins, i.e. $N_{\Gamma_2^c}(v) \neq N_{\Gamma_2^c}(w)$ for all $v, w \in V(\Gamma_2^c)$.
\end{enumerate}
Now, starting with the graphs $\Gamma_1$ and $\Gamma_2$, Lemma \ref{one_vertex_add} together with Remark \ref{Rem_Ban_one_vertex} allow us to construct the connected graphs $\Gamma_1'$ and $\Gamma_2'$ such that $QAut_{Ban}(\Gamma_i) \cong QAut_{Ban}(\Gamma_i')$ for $i=1,2$. Next, applying Lemma \ref{n_vertex_add} and Remark \ref{Rem_Ban_n vertex_add} to the graphs $\Gamma_i'$, one can obtain connected graphs $\widetilde{\Gamma_i'}$ such that $QAut_{Ban}(\widetilde{\Gamma_i'}) \cong QAut_{Ban}(\Gamma_i')$ and whose complement graphs ${\widetilde{\Gamma_i'}}^c$ are also connected. Therefore, since the aforementioned conditions (i) and (ii) hold vacuously, we have $QAut_{Ban}(\widetilde{\Gamma_1'} \circ \widetilde{\Gamma_2'}) \cong QAut_{Ban}(\widetilde{\Gamma_1'}) \wr_{*} QAut_{Ban}(\widetilde{\Gamma_2'}) \cong QAut_{Ban}(\Gamma_{1}) \wr_{*} QAut_{Ban}(\Gamma_{2}) $. \\

\noindent However, the analogue of Sabidussi’s result fails in the setting of Bichon’s quantum automorphism groups. For instance, $$QAut_{Bic}(K_2^c \circ K_2) \ncong QAut_{Bic}(K_2^c) \wr_{*} QAut_{Bic}(K_2);$$ since $QAut_{Bic}(K_2^c \circ K_2)$ is commutative, whereas $QAut_{Bic}(K_2^c) \wr_{*} QAut_{Bic}(K_2) \cong C(H_2^+)$ is non-commutative. Hence, the similar construction for the free wreath product is not valid in general for the quantum automorphism group of graphs in Bichon's setup. To construct such a graph $\Gamma_{wr}$, we instead use the corona product and show that, similar to the lexicographic product, a similar isomorphism holds for the corona product of two graphs under a weaker hypothesis, but this result is valid in both Banica’s and Bichon’s frameworks. More precisely, in the following theorem, we establish a quantum analogue of Theorem 1 (or its Corollary) in \cite{corona}, showing that the free wreath product of the quantum automorphism groups of two graphs $\Gamma_2$ and $\Gamma_1$, namely $QAut_{Ban/Bic}(\Gamma_2) \wr_{*} QAut_{Ban/Bic}(\Gamma_1)$, arises as the quantum automorphism group $QAut_{Ban/Bic}(\Gamma_1 \odot \Gamma_2) $ of the corona product $\Gamma_1 \odot \Gamma_2$.

\begin{thm} \label{QAut_Corona_prod}
Let $\Gamma_1$ and $\Gamma_2$ be two finite graphs such that $\Gamma_1$ or $\Gamma_2^c$ has no isolated vertices. Then
$$QAut_{Ban}(\Gamma_1 \odot \Gamma_2) \cong QAut_{Ban}(\Gamma_2) \wr_{*} QAut_{Ban}(\Gamma_1)$$ and $$QAut_{Bic}(\Gamma_1 \odot \Gamma_2) \cong QAut_{Bic}(\Gamma_2) \wr_{*} QAut_{Bic}(\Gamma_1).$$
\end{thm}	
\begin{proof}
	We will establish the isomorphism in Bichon's framework; however, a similar argument also concludes the same isomorphism in Banica's setting.
	Let us denote the canonical fundamental representation of $QAut_{Bic}(\Gamma_1)$, $QAut_{Bic}(\Gamma_1)$, and  $QAut_{Bic}(\Gamma_1 \odot \Gamma_2)$ by $u$, $v$, and $w$ respectively. We also set $n:=|V(\Gamma_1)|$ and $m:=|V(\Gamma_2)|$. 
	Firstly, we show there exists a surjective *-homomorphism $\phi:QAut_{Bic}(\Gamma_1 \odot \Gamma_2) \to QAut_{Bic}(\Gamma_2) \wr_{*} QAut_{Bic}(\Gamma_1)$ such that $w \mapsto \begin{bmatrix}
	u & 0_{n,nm}\\
	 0_{nm,n} & v \wr_{*} u \\
	\end{bmatrix} =:w'$, where $v \wr_{*} u$ denotes the matrix of order $(nm \times nm)$ defined by $(v \wr_{*} u)_{(i, \alpha)(j, \beta)} := v^{(\alpha)}_{ij}u_{\alpha \beta}$. Explicitly, $w_{\alpha \beta} \mapsto u_{\alpha \beta}$, $w_{(i, \alpha)(j, \beta)} \mapsto v^{(\alpha)}_{ij}u_{\alpha \beta}$ and $w_{\alpha (i, \beta)} \mapsto 0$ for all $\alpha, \beta \in V(\Gamma_1)$ and $i,j \in V(\Gamma_2)$. Note that $w'$ is a magic unitary matrix. We claim that $w'$ commutes with the adjacency matrix $A_{\Gamma_1 \odot \Gamma_2}$ of $\Gamma_1 \odot \Gamma_2$, where $$A_{\Gamma_1 \odot \Gamma_2}:= \begin{bmatrix}
	A_{\Gamma_1} & Q_1 & Q_2 & \cdots & Q_n \\
	Q_1^t & A_{\Gamma_2} & 0_{m}  & \cdots & 0_{m}  \\
	Q_2^t & 0_{m} & A_{\Gamma_2} & \cdots & 0_{m}  \\
	\vdots & \vdots & \vdots & \ddots & \vdots\\
	Q_n^t & 0_{m} & 0_{m} & \cdots & A_{\Gamma_2}
	\end{bmatrix}= \begin{bmatrix}
	A_{\Gamma_1} & (J_{1,m} \otimes I_n) \\
	(J_{1,m} \otimes I_n)^t & A_{\Gamma_2} \otimes I_n \\
	\end{bmatrix}$$ and $Q_k$ is an $n \times m$ matrix in which all the entries in the $k$-th row are equal to 1, while all other entries are 0. The commutation relation is equivalent to the equations:
	\begin{align}
 uA_{\Gamma_1} & = A_{\Gamma_1}u, \label{Corona_eq1}\\
\text{and \ \ \ }  (v \wr_{*} u)(A_{\Gamma_2} \otimes I_n) & = (A_{\Gamma_2} \otimes I_n)(v \wr_{*} u), \label{Corona_eq2}
\end{align}
since the magic unitary properties of $u$ and $v$ always ensure the following equalities: 
\begin{align}
 u \begin{bmatrix}
Q_1 & Q_2 & ...& Q_n
\end{bmatrix} & = \begin{bmatrix}
Q_1 & Q_2 & ...& Q_n
\end{bmatrix}(v \wr_{*} u), \label{Corona_eq3}\\
 (v \wr_{*}u)\begin{bmatrix}
	Q_1 & Q_2 & ...& Q_n
\end{bmatrix}^t & = 
\begin{bmatrix}
	Q_1 & Q_2 & ...& Q_n
\end{bmatrix}^t u. \label{Corona_eq4}
	\end{align}	
To verify this, computing the left-hand side (LHS) of \eqref{Corona_eq3}, we get an $n \times mn$ block matrix $$
	[~\underbrace{C_1(u) ~  ... ~ C_1(u)}_{m -times}  ~|~ \underbrace{C_2(u) ~ ... ~ C_2(u)}_{m-times} ~| ~  \hspace{0.3cm} ... ~ \hspace{0.3cm}  ~|~ \underbrace{C_n(u) ~ ... ~  C_n(u)}_{m-times}~], $$ where $C_k(u)$ denotes the $k$-th column of $u$. On the other hand, the right-hand side (RHS) of Equation \eqref{Corona_eq3} is an $n \times mn$ block matrix of the form $$ \begin{bmatrix}
	B_1 & B_2 & ... & B_n
	\end{bmatrix},$$ where each $B_k$ denotes an $m \times m$ matrix. Since $V^{(l)}=(v^{(l)}_{ij})_{m \times m}$ is a magic unitary matrix for each $l \in [n]$, we have $Q_l V^{(l)}= Q_l$, and hence $$B_k= \sum\limits_{l=1}^{m} Q_{l}V^{(l)}u_{lk} = \sum\limits_{l=1}^{m} Q_l u_{lk}= \begin{bmatrix}
	C_k(u) & C_{k}(u) & ... & C_{k}(u)
	\end{bmatrix}.$$ Therefore, the LHS and RHS are equal in Equation \eqref{Corona_eq3}. Clearly, Equation \eqref{Corona_eq4} follows from \eqref{Corona_eq3} using the fact that both $u$ and $v \wr_{*} u$ are self-adjoint and magic unitary matrices.\\
	Finally, we verify defining relation \eqref{Bic_comm_relation} of $QAut_{Bic}(\Gamma_1 \odot \Gamma_2)$: \\
	$\bullet$ If $(\alpha, \beta), (\gamma, \delta) \in E(\Gamma_1 \odot \Gamma_2)$, then $w'_{\alpha \gamma} = u_{\alpha \gamma}$ commutes with $u_{\beta \delta} = w'_{\beta \delta}$. \\
	$\bullet$ If $((i, \alpha), (j, \alpha)) \in E(\Gamma_1 \odot \Gamma_2)$ and $((k, \gamma), (l, \gamma)) \in E(\Gamma_1 \odot \Gamma_2)$, then $(i,j),(k,l) \in E(\Gamma_2)$, which implies  
	$v^{(\alpha)}_{ik} v^{(\alpha)}_{jl} = v^{(\alpha)}_{jl}v^{(\alpha)}_{ik}   ~~ \forall \alpha \in V(\Gamma_1)$. Multiplying both sides by $u_{\alpha \gamma}^2$, we obtain:
	\begin{align*}
	& v^{(\alpha)}_{ik} v^{(\alpha)}_{jl} u_{\alpha \gamma} u_{\alpha \gamma}= v^{(\alpha)}_{jl}v^{(\alpha)}_{ik} u_{\alpha \gamma} u_{\alpha \gamma} \\
	\implies & w'_{(i,\alpha)(k, \gamma)} w'_{(j, \alpha)(l, \gamma)}=v^{(\alpha)}_{ik} u_{\alpha \gamma} v^{(\alpha)}_{jl}  u_{\alpha \gamma} = v^{(\alpha)}_{jl} u_{\alpha \gamma} v^{(\alpha)}_{ik} u_{\alpha \gamma}=w'_{(j, \alpha)(l, \gamma)} w'_{(i,\alpha)(k, \gamma)}.
	\end{align*}
$\bullet$ If $(\alpha, (i, \alpha)), (\beta, (j, \beta)) \in E(\Gamma_1 \odot \Gamma_2)$, then $w'_{\alpha \beta} = u_{\alpha \beta}$	clearly commutes with $v^{(\alpha)}_{ij}u_{\alpha \beta} = w'_{(i, \alpha)(j,\beta)}$.\\
$\bullet$ In all other cases, $w'_{\alpha (k, \beta)}=0$ for all $\alpha, \beta \in V(\Gamma_1)$ and $k \in V(\Gamma_2)$. So, the other commutativity relations hold trivially.\\

For the converse direction, observe that $deg_{\Gamma_1 \odot \Gamma_2}(\alpha)=deg_{\Gamma_1}(\alpha)+|V(\Gamma_2)| > m$ (if $\Gamma_{1}$ has no isolated vertices) and $deg_{\Gamma_1 \odot \Gamma_2}(i,\beta)=deg_{\Gamma_2}(i)+1 < m$ (if $\Gamma_{2}^c$ has no isolated vertices) for all $i \in V(\Gamma_2)$, $\alpha, \beta \in V(\Gamma_1) $. By the hypothesis of the theorem, $deg(\alpha) \neq deg(i, \beta)$ for all $i \in V(\Gamma_2)$ and $\alpha, \beta \in V(\Gamma_1)$. Therefore, Proposition \ref{3} implies that $w_{\alpha (i, \beta)}=0$. Now, $w$ reduces to the block diagonal matrix $\begin{bmatrix}
X & 0\\
0 & Y
\end{bmatrix}$ (where $X$ and $Y$ are the magic unitary matrices of order $n \times n$ and $nm \times nm$ respectively), which commutes with $A_{\Gamma_1 \odot \Gamma_2}$. This yields the following set of equations:
\begin{align}
XA_{\Gamma_1} &= A_{\Gamma_1}X, \label{Corona_eq1'}\\
Y(A_{\Gamma_2} \otimes I_n) &= (A_{\Gamma_2} \otimes I_n)Y, \label{Corona_eq2'}\\
X \begin{bmatrix}
Q_1 & Q_2 & ...& Q_n
\end{bmatrix} &= \begin{bmatrix}
Q_1 & Q_2 & ...& Q_n
\end{bmatrix}Y, \label{Corona_eq3'}\\
Y \begin{bmatrix}
Q_1 & Q_2 & ...& Q_n
\end{bmatrix}^t &= 
\begin{bmatrix}
Q_1 & Q_2 & ...& Q_n
\end{bmatrix}^t X  \label{Corona_eq4'}
\end{align}
which are structurally similar to Equations \eqref{Corona_eq1} - \eqref{Corona_eq4}.
Equations \eqref{Corona_eq1'} and \eqref{Corona_eq2'} show that $X$ and $Y$ satisfy the defining conditions of $QAut_{Ban}(X)$ and $QAut_{Ban}(n \Gamma_2)$ respectively. Moreover, from the edge structure of $\Gamma_1 \odot \Gamma_2$, it is straightforward to verify that $X$ and $Y$ satisfy the commutativity relation \eqref{Bic_comm_relation} of $QAut_{Bic}(\Gamma_1)$ and $QAut_{Bic}(n \Gamma_2)$ respectively.  
Since it is known from Theorem \ref{disjoint_union_k Gamma} that $QAut_{Bic}(n\Gamma_2) \approx QAut_{Bic}(\Gamma_2) \wr_{*} S_n^+$ with respect to their canonical fundamental representations, there exists a surjective *-homomorphism $$\theta: QAut_{Bic}(\Gamma_1) * [QAut_{Bic}(\Gamma_2) \wr_{*} S_n^+] \to QAut_{Bic}(\Gamma_1 \odot \Gamma_2) $$ which maps $$\begin{bmatrix}
u & 0\\
0 & v \wr_{*}t
\end{bmatrix} \mapsto w= \begin{bmatrix}
	X & 0\\
	0 & Y
\end{bmatrix},$$ where $(v \wr_{*} t)_{nm \times nm}$ is a block matrix $(V^{(\alpha)} t_{\alpha \beta})_{n \times n}$ with each block $V^{(\alpha)} t_{\alpha \beta}:=(v^{(\alpha)}_{ij}t_{\alpha \beta})_{m \times m}$,  and $t:=(t_{\alpha \beta})_{n \times n}$ is the canonical fundamental representation of $S_n^+$. Let $J$ denote the ideal generated by $\{u_{\alpha \beta} - t_{\alpha \beta}\}_{\alpha, \beta \in V(\Gamma_1)}$. Now, it is easy to verify the quotient algebra $$\left( QAut_{Bic}(\Gamma_1) * [QAut_{Bic}(\Gamma_2) \wr_{*} S_n^+] \right)/J \approx  QAut_{Bic}(\Gamma_1) \wr_{*} QAut_{Bic}(\Gamma_1)$$ via isomorphism $$\bar{u}_{\alpha \beta} \mapsto u_{\alpha \beta} \text{ \hspace{1cm} and \hspace{1cm}} \bar{v}^{(\alpha)}_{ij} \bar{t}_{\alpha \beta} \mapsto v^{(\alpha)}_{ij} u_{\alpha \beta},$$ where $\bar{u}_{\alpha \beta}, \bar{v}_{\alpha \beta}$ and $\bar{t}_{\alpha \beta}$ denotes the image of $u_{\alpha \beta},v_{\alpha \beta}$ and $t_{\alpha \beta}$ under the canonical  quotient map $$\pi: \left( QAut_{Bic}(\Gamma_1) * [QAut_{Bic}(\Gamma_2) \wr_{*} S_n^+] \right) \to \left( QAut_{Bic}(\Gamma_1) * [QAut_{Bic}(\Gamma_2) \wr_{*} S_n^+] \right)/J $$ with $\bar{u}_{\alpha \beta}=\bar{t}_{\alpha \beta}$. 
Moreover, a computation analogous to that showing LHS = RHS in Equation \eqref{Corona_eq3}, here also Equation \eqref{Corona_eq3'}, together with the fact that $\theta$ is a *-homomorphism, ensures that $\theta (u_{\alpha \beta}) = \theta (t_{\alpha \beta})$. Hence, $J \subseteq ker \theta$. By Theorem 2.11 of \cite{Wangfree}, there exists a surjective *-homomorphism $$\bar{\theta}:QAut_{Bic}(\Gamma_2) \wr_{*} QAut_{Bic}(\Gamma_1) \to QAut_{Bic}(\Gamma_1 \odot \Gamma_2)$$ such that $$\bar{\theta}(u_{\alpha \beta})=X_{\alpha \beta} \text{ \hspace{1cm} and \hspace{1cm}} \bar{\theta}(v^{(\alpha)}_{ij} u_{\alpha \beta})= Y_{(i, \alpha)(j, \beta)}.$$ 
Furthermore, only the entries of $Y$ generate $QAut_{Bic}(\Gamma_1 \odot \Gamma_2)$ because $$X_{\alpha \beta}=\bar{\theta}(u_{\alpha \beta}) =\bar{\theta} \left(\sum\limits_{j=1}^{m} v^{(\alpha)}_{ij} u_{\alpha \beta} \right)= \sum\limits_{j=1}^{m}Y_{(i, \alpha)(j, \beta)}.$$ 
Therefore, the surjective *-homomorphism $\bar{\theta}$ maps the generating matrix $v \wr_{*}u \mapsto Y$ such that $\phi \circ \bar{\theta}= id_{QAut_{Bic}(\Gamma_2) \wr_{*} QAut_{Bic}(\Gamma_1)}$ and $\bar{\theta} \circ \phi = id_{QAut_{Bic}(\Gamma_1 \odot \Gamma_2)}$, which completes the proof.\\
\end{proof}

\noindent Before proving the main result in this section, observe that if $\Gamma_1$ is connected, then $\Gamma_1 \odot \Gamma_2$ is also connected.
\begin{thm} \label{wreath_prod_construction}
	Given two finite graphs $\Gamma_1$ and $\Gamma_2$, there exists a finite connected graph $\Gamma_{wr}$ such that $QAut_{Bic}(\Gamma_{wr}) \cong QAut_{Bic}(\Gamma_2) \wr_{*} QAut_{Bic}(\Gamma_1)$.
\end{thm}
\begin{proof}
	If $\Gamma_1$ is connected, then by Theorem \ref{QAut_Corona_prod}, we may simply take $\Gamma_{wr} := \Gamma_1 \odot \Gamma_2$.\\
	If $\Gamma_1$ is disconnected, then using Lemma \ref{one_vertex_add}, we can construct a connected graph $\Gamma_1'$ such that $QAut_{Bic}(\Gamma_1') \cong QAut_{Bic}(\Gamma_1)$. Therefore, the connected graph $\Gamma_{wr}:= \Gamma_1' \odot \Gamma_2$ satisfies the desired isomorphism.
\end{proof}

\begin{rem}
	Theorem \ref{QAut_Corona_prod} and Remark \ref{Rem_Ban_one_vertex} ensure that $QAut_{Ban}(\Gamma_{wr}) \cong QAut_{Ban}(\Gamma_2) \wr_{*} QAut_{Ban}(\Gamma_1)$.
\end{rem}

\section*{Acknowledgements}
\noindent  The first author acknowledges the financial support from the University Grants Commission (UGC), India (NTA Ref. No.: 241610170358). The second author acknowledges the financial support from the Department of Science and Technology, India (DST/INSPIRE/03/2021/001688). All the authors also acknowledge the support from DST-FIST (File No. SR/FST/MS-I/2019/41). The authors are also grateful to Angsuman Das for valuable discussions on graph theory.

\end{document}